%% file: main.tex
\documentclass[11pt,onecolumn]{article}

\usepackage{a4wide}
\usepackage{amsmath,amsthm,epsfig,amssymb,amsbsy}
\usepackage{enumerate}
\usepackage{comment}
\usepackage{algorithm}
\usepackage{amsfonts}       %
\usepackage{dsfont}
\usepackage[inline]{enumitem}
\usepackage{amssymb}
\usepackage{mathtools}
\usepackage{cases}

\usepackage{algorithmic}

\usepackage{nicefrac}

\usepackage{subcaption}
\usepackage{graphicx}

\usepackage{pgfplots}
\pgfplotsset{compat=1.18}

\usepackage{pgfplotstable}
\usepackage{multirow}

\usepackage{adjustbox}

\usepackage[giveninits=true,maxbibnames=9,maxcitenames=2,backend=biber,url=false,doi=false,isbn=false]{biblatex}
\newenvironment{keywords}{\begin{paragraph}{Keywords:}}{\end{paragraph}}

\DeclareMathOperator*\minimize{minimize}

\DeclareMathOperator{\dom}{dom}

\DeclareMathOperator{\intr}{int}

\DeclareMathOperator{\id}{id}
\DeclareMathOperator{\barsgn}{\overline{sgn}}
\DeclareMathOperator{\arcsinh}{arcsinh}

\usepackage{booktabs,dcolumn,caption}
\newcolumntype{d}[1]{D{.}{.}{#1}} %

\newcommand{\bR}{\mathbb{R}}

\newcommand{\bN}{\mathbb{N}}

\newcommand{\exR}{\overline{\mathbb{R}}}

\newcommand{\cC}{\mathcal{C}}

\newcommand{\cI}{\mathcal{I}}

\newcommand{\cB}{\mathcal{B}}

\newcommand{\cO}{\mathcal{O}}

\newcommand{\cS}{\mathcal{S}}

\newcommand\downto{\searrow}%

\newcommand\R{\mathbb{R}}
\newcommand\N{\mathbb{N}}

\newcommand\eqdef{:=}

\newcommand{\Cpp}{C\kern-0.04em+\kern-0.01em+}

\makeatletter
\newcommand{\aprox}[3][\@nil]{%
  \def\tmp{#1}%
   \ifx\tmp\@nnil
       \operatorname{prox}_{#3}^{#2}
    \else
         \operatorname{prox}_{#3}^{#1 \star #2}
    \fi}

\newcommand{\aenv}[3][\@nil]{%
  \def\tmp{#1}%
   \ifx\tmp\@nnil
       \operatorname{env}_{#3}^{#2}
    \else
         \operatorname{env}_{#3}^{#1 \star #2}
    \fi}

\newcommand{\bprox}[3][\@nil]{%
  \def\tmp{#1}%
   \ifx\tmp\@nnil
       \operatorname{bprox}_{#3}^{#2}
    \else
        \operatorname{bprox}_{#3}^{#1 #2}
    \fi}
\makeatother

\usepackage{hyperref}
\hypersetup{
  colorlinks=true,
  linkcolor=blue,
  filecolor=magenta,
  citecolor =magenta,  
  urlcolor=magenta,
  pdftitle={Newton methods beyond Hessian Lipschitz continuity}
}
\usepackage[]{cleveref}[0.19]

\crefname{section}{Section}{Sections}
\crefname{subsection}{Subsection}{Subsections}
\Crefname{section}{Section}{Sections}
\Crefname{subsection}{Subsection}{Subsections}

\Crefname{figure}{Figure}{Figures}

\crefformat{equation}{\textup{#2(#1)#3}}
\crefrangeformat{equation}{\textup{#3(#1)#4--#5(#2)#6}}
\crefmultiformat{equation}{\textup{#2(#1)#3}}{ and \textup{#2(#1)#3}}
{, \textup{#2(#1)#3}}{, and \textup{#2(#1)#3}}
\crefrangemultiformat{equation}{\textup{#3(#1)#4--#5(#2)#6}}%
{ and \textup{#3(#1)#4--#5(#2)#6}}{, \textup{#3(#1)#4--#5(#2)#6}}{, and \textup{#3(#1)#4--#5(#2)#6}}

\Crefformat{equation}{#2Equation~\textup{(#1)}#3}
\Crefrangeformat{equation}{Equations~\textup{#3(#1)#4--#5(#2)#6}}
\Crefmultiformat{equation}{Equations~\textup{#2(#1)#3}}{ and \textup{#2(#1)#3}}
{, \textup{#2(#1)#3}}{, and \textup{#2(#1)#3}}
\Crefrangemultiformat{equation}{Equations~\textup{#3(#1)#4--#5(#2)#6}}%
{ and \textup{#3(#1)#4--#5(#2)#6}}{, \textup{#3(#1)#4--#5(#2)#6}}{, and \textup{#3(#1)#4--#5(#2)#6}}

\newtheorem{theorem}{Theorem}[section]
\makeatletter
\newcommand{\settheoremtag}[1]{%
  \let\oldthetheorem\thetheorem%
  \renewcommand{\thetheorem}{#1}%
  \g@addto@macro\endtheorem{%
    \addtocounter{theorem}{-1}%
    \global\let\thetheorem\oldthetheorem}%
  }
\makeatother

\newtheorem{lemma}[theorem]{Lemma}
\newlist{lemenum}{enumerate}{1} %
\setlist[lemenum]{label=(\roman*), ref=\theproposition(\roman*), font=\rm}
\crefalias{lemenumi}{lemma}

\newlist{propenum}{enumerate}{1} %
\setlist[propenum]{label=(\roman*), ref=\theproposition(\roman*), font=\rm}
\crefalias{propenumi}{proposition} 

\newlist{theoremenum}{enumerate}{1} %
\setlist[theoremenum]{label=(\roman*), ref=\thetheorem(\roman*), font=\rm}
\crefalias{theoremenumi}{theorem} 

\newtheorem{definition}[theorem]{Definition}

\newtheorem{assumption}[theorem]{Assumption}

\newtheorem{example}[theorem]{Example}

\newtheorem{remark}[theorem]{Remark}

\newlist{assumenum}{enumerate}{1} %
\setlist[assumenum]{label=(\roman*), ref=\thetheorem(\roman*), font=\rm}
\crefalias{assumenumi}{assumption} 

\crefname{theorem}{Theorem}{Theorems}
\Crefname{theorem}{Theorem}{Theorems}
\crefname{lemma}{Lemma}{Lemmas}
\Crefname{lemma}{Lemma}{Lemmas}
\crefname{proposition}{Proposition}{Propositions}
\Crefname{proposition}{Proposition}{Propositions}
\crefname{corollary}{Corollary}{Corollaries}
\Crefname{corollary}{Corollary}{Corollaries}
\crefname{definition}{Definition}{Definitions}
\Crefname{definition}{Definition}{Definitions}
\crefname{fact}{Fact}{Facts}
\Crefname{fact}{Fact}{Facts}
\crefname{remark}{Remark}{Remarks}
\Crefname{remark}{Remark}{Remarks}
\crefname{example}{Example}{Examples}
\Crefname{example}{Example}{Examples}
\crefname{assumption}{Assumption}{Assumptions}
\Crefname{assumption}{Assumption}{Assumptions}
\crefname{algorithm}{Algorithm}{Algorithms}
\Crefname{algorithm}{Algorithm}{Algorithms}

\usepackage{xcolor}

\usepackage{microtype}

\usepackage{crossreftools}
\title{Newton methods beyond Hessian Lipschitz continuity:\\ A nonlinear preconditioning approach}
\author{Alexander Bodard\thanks{KU Leuven,
		Department of Electrical Engineering (ESAT-STADIUS),
		Kasteelpark Arenberg 10, 3001 Leuven, Belgium~
    {\tt%
            \href{mailto:alexander.bodard@esat.kuleuven.be}{\{alexander.bodard,}%
      			\href{mailto:panos.patrinos@esat.kuleuven.be}{panos.patrinos\}}%
			\href{mailto:alexander.bodard@esat.kuleuven.be,panos.patrinos@esat.kuleuven.be}{@esat.kuleuven.be}%
		}\\
    The authors are supported by the Research Foundation Flanders (FWO) research project G033822N and the Research Council KU Leuven C1 project with ID C14/24/103.
} \and Panagiotis Patrinos$^*$
}
 
\begin{document}

\maketitle

\begin{abstract}

\input{abstract.tex}
\end{abstract}


\tableofcontents

\input{contents.tex}

\newpage
\printbibliography

\newpage
\appendix
\input{appendix.tex}

\end{document}

%% file: abstract.tex
Newton-type methods are typically analyzed under Lipschitz continuity of the
Hessian, an assumption that can fail for objectives with higher-order or
polynomial growth. We introduce a class of nonlinearly preconditioned Newton
methods that apply Newton's root-finding scheme to a transformed optimality
mapping, thereby extending recent nonlinear
preconditioning ideas from first-order methods to the second-order setting.
The resulting methods are naturally analyzed under Lipschitz continuity of a
preconditioned Hessian, a condition that significantly relaxes the classical
Hessian Lipschitz continuity assumption. Under this generalized smoothness model, we
establish local superlinear and quadratic convergence guarantees, and develop a
globalization strategy for the nonregularized method despite the fact that the
preconditioned Newton direction need not be a descent direction. We further
propose a regularized variant for isotropic preconditioners, and show that it attains an
$\cO(\varepsilon^{-3/2})$ iteration complexity. An adaptive version removes the
need to know the smoothness constant and allows inexact subproblem solutions
while preserving the same complexity order. 

%% file: contents.tex
\section{Introduction} \label{sec:intro}
\input{contents/introduction.tex}

\section{Nonlinearly preconditioned Newton} \label{sec:pn}
\input{contents/preconditioned_newton.tex}

\section{Regularization} \label{sec:cubic-pn}
\input{contents/cubic-regularization.tex}

\section{Numerical results}

\input{contents/numerics.tex}

\section{Conclusion} \label{sec:conclusion}

\input{contents/conclusion.tex}

%% file: contents/introduction.tex
We study unconstrained optimization problems of the form
\begin{equation} \label{eq:problem} \tag{P}
    \minimize_{x \in \R^n} f(x),
\end{equation}
where \(f : \R^n \to \R\) is a twice continuously differentiable and possibly nonconvex function for which \(\inf f > -\infty\).
Gradient descent constitutes an efficient solver for this class of problems, but oftentimes requires carefully tuned stepsizes to converge, and may suffer from problems like exploding gradients.
In response to this, practitioners developed various adaptive gradient schemes, like Adagrad \cite{duchi2011adaptive} and Adam \cite{kingma2014adam}, as well as gradient clipping techniques \cite{zhang_why_2020}.
It was recently shown \cite[\S 1.6]{oikonomidis_nonlinearly_2025} that all of these methods are encapsulated within a single framework of \emph{nonlinearly preconditioned gradient methods}, which, for a stepsize \(\gamma > 0\), has iterations
\begin{equation} \label{eq:pg} \tag{PG}
    x^{k+1} = x^k - \gamma \nabla \phi^*(\nabla f(x^k)).
\end{equation}
Here, \(\phi : \R^n \to \exR\) is referred to as the \emph{reference function} and \(\phi^*\), its convex conjugate, is the \emph{dual reference function}.
Note that \(\phi = \frac{1}{2} \Vert \cdot \Vert^2\) recovers classical gradient descent.

Under a generalized smoothness assumption, \eqref{eq:pg} can be analyzed from a majorization-minimization perspective, in the sense that at every iteration an upper bound of the form 
\begin{equation} \label{eq:majorizer}
    f(x) \leq f(x^k) + \tfrac{1}{L} \phi(L(x - y^k)) - \tfrac{1}{L} \phi(L(x^{k} - y^k))
\end{equation}
is minimized with respect to \(x\), where \(y^k = x^k - \frac{1}{L} \nabla \phi^*(\nabla f(x^k))\) and \(L = \nicefrac{1}{\gamma}\).
This inequality corresponds to a natural extension of Lipschitz smoothness, called \emph{anisotropic smoothness} \cite{laude_anisotropic_2025}.
Notably, it enables the construction of upper bounds that more tightly fit the properties of the objective function \(f\), since \(\phi\) may be chosen to grow faster than a quadratic function. 
Moreover, anisotropic smoothness encompasses the established notion of \((L_0, L_1)\)-smoothness \cite{zhang_why_2020}, and was recently shown to hold for key applications such as matrix factorization and phase retrieval \cite{bodard2025escaping}.

In this work, we propose a novel class of Newton methods for solving \eqref{eq:problem} that serves as a second-order counterpart to \eqref{eq:pg}, and we study its properties under anisotropic smoothness and related notions. 
Recall that Newton's method performs iterations \(x^{k+1} = x^k + d^k\), where the update directions satisfy
\begin{equation} \label{eq:n} \tag{N}
    \nabla^2 f(x^k) d^k = - \nabla f(x^k).
\end{equation}
Instead, we consider a generalization of this method where
\begin{equation} \label{eq:pn} \tag{PN}
    \nabla^2 \phi^*(\nabla f(x^k)) \nabla^2 f(x^k) d^k = - \nabla \phi^*(\nabla f(x^k)),
\end{equation}
This is nothing else than Newton's root-finding method applied to \(\nabla \phi^* \circ \nabla f\) instead of \(\nabla f\), and is justified under appropriate assumptions on the reference function \(\phi\) that guarantee for any \(x^\star \in \R^n\) the equivalence
\begin{equation} \label{eq:preconditioned-gradient}
    \nabla f(x^\star) = 0 \quad \Longleftrightarrow \quad \nabla \phi^*(\nabla f(x^\star)) = 0.
\end{equation}
Remark that Newton's method is \emph{affine invariant}, so linear preconditioning does not modify the update direction.
In contrast, the \emph{nonlinear preconditioning} \eqref{eq:pn} does.

By designing Newton variants based on generalized smoothness conditions that better reflect the properties of the cost function, we hope to obtain novel algorithms with improved practical performance, and, more generally, to narrow the gap between optimization theory and practice.

\paragraph{Contributions} Our contributions can be summarized as follows.

\begin{itemize}
    \item We propose a class of preconditioned Newton methods of which the update directions satisfy \eqref{eq:pn}.
    This scheme is naturally analyzed under a Lipschitz continuity assumption on the \emph{preconditioned} Hessian (cfr.\,\cref{as:prec-hessian-lipschitz}). 
    For appropriate reference functions this may relax the usual Hessian Lipschitz continuity condition.
    Notably, we show that \cref{as:prec-hessian-lipschitz} holds for any 1D polynomial objective.
    We characterize the region of quadratic convergence of \eqref{eq:pn} in terms of the potentially more favorable Lipschitz constant of the \emph{preconditioned Hessian}.
    Moreover, a globalization strategy of the scheme is proposed, which is non-trivial due to the update directions not being descent directions in general. 
    \item Restricting ourselves to isotropic reference functions, we then present a regularized variant that attains an \(\cO(\varepsilon^{-\nicefrac{3}{2}})\) complexity under this generalized smoothness condition.
    Well-definedness of the subproblems is described in detail.
    We also present an adaptive variant which does not require explicit knowledge of the smoothness constant and handles inexact subproblem solutions, while maintaining the same complexity.
\end{itemize}

\paragraph{Notation} 
By \(\langle \cdot, \cdot \rangle\) we denote the Euclidean inner product on \(\R^n\), and by \(\Vert \cdot \Vert\) the induced Euclidean norm on \(\R^n\) as well as the spectral norm for matrices.
For a square matrix \(A\) with real spectrum, the largest and smallest eigenvalues are denoted by \(\lambda_{\max}(A)\) and \(\lambda_{\min}(A)\), respectively.
Likewise, \(\sigma_{\max}(A)\) and \(\sigma_{\min}(A)\) refer to the largest and smallest singular values.
We use the notation \(\barsgn(x) = \nicefrac{x}{\Vert x \Vert}\) for \(x \in \R^n \setminus \{0\}\) and \(0\) otherwise.
The class of \(k\) times continuously differentiable functions is denoted by \(\cC^k\).
A function \(f \in \cC^1\) is \(L_{f,1}\)-Lipschitz smooth or has \(L_{f,1}\)-Lipschitz continuous gradients if \(\Vert \nabla f(x) - \nabla f(y) \Vert \leq L_{f, 1} \Vert x - y \Vert\) for all \(x, y \in \R^n\) with \(L_{f,1}\geq 0\).
A function \(f \in \cC^2\) has \(L_{f,2}\)-Lipschitz continuous Hessians if \(\Vert \nabla^2 f(x) - \nabla^2 f(y) \Vert \leq L_{f, 2} \Vert x - y \Vert\) for all \(x, y \in \R^n\) with \(L_{f,2}\geq 0\), and is \((L_0, L_1)\)-smooth if \(\Vert \nabla^2 f(x) \Vert \leq L_0 + L_1 \Vert \nabla f(x) \Vert\) for all \(x \in \R^n\) with \(L_0, L_1 \geq 0\).
Otherwise, we adopt the notation from \cite{rockafellar_variational_1998}.

\subsection{Related work}

\paragraph{Nonlinear Preconditioning} The nonlinearly preconditioned gradient method \eqref{eq:pg} was first analyzed in the convex setting by \cite{maddison_dual_2021}, and later in the nonconvex setting \cite{laude_anisotropic_2025,oikonomidis_nonlinearly_2025} under anisotropic smoothness, a notion introduced in \cite{laude_dualities_2023} and closely related to the concept of \(\Phi\)-convexity \cite{leger_gradient_2023,oikonomidis_forward-backward_2025}.
The method has also been extended to measure spaces \cite{bonet_mirror_2024}.
Recent works presented variants with momentum and a stochastic analysis \cite{oikonomidis2025nonlinearly}, and established saddle point avoidance of the scheme \cite{bodard2025escaping}. 

\paragraph{Generalized smoothness}

Gradient-based methods are traditionally analyzed under restrictive Lipschitz smoothness assumptions, see e.g.\,\cite{nesterov_lectures_2018}.
Bregman relative smoothness is a milder condition that allows for unbounded Hessians \cite{lu_relatively_2018}.
More recently, the \((L_0, L_1)\)-smoothness condition, proposed by \cite{zhang_why_2020} based on empirical observations in training LSTMs, has gained traction.
Further generalizations include \(\alpha\)-symmetric smoothness \cite{chen_generalized-smooth_2023} and \(\ell\)-smoothness \cite{li_convex_2023}.
Remark that anisotropic smoothness subsumes \((L_0, L_1)\)-smoothness and \(\ell\)-smoothness for certain reference functions \cite{bodard2025escaping}.
For Newton methods, Lipschitz continuity of the Hessian or self-concordance are often assumed \cite{nesterov1994interior,hanzely2022damped,doikov2025minimizing}.

\paragraph{Newton's method} Following the terminology of \cite{maddison_dual_2021}, the proposed method \eqref{eq:pn} is derived by `left'-preconditioning \(\nabla f\), cfr.\,\eqref{eq:preconditioned-gradient}. A distinct class of `right'-preconditioned Newton methods has been studied in \cite{burachik2012generalized,burachik2021generalized}.
Various globalization and regularization techniques have been explored for Newton methods.
Cubic regularization methods \cite{nesterov2006cubic} attain an \(\cO \left(\varepsilon^{-\nicefrac{3}{2}}\right)\) complexity under Lipschitz continuity of the Hessian. This is optimal for this class of problems \cite{carmon2020lower}.
Following the works of \cite{cartis2011adaptive,cartis2011adaptive2}, \emph{adaptive} variants that attain the optimal rate without knowledge of the actual Lipschitz constant are a popular subject of study \cite{dussault2024scalable}, and also \emph{universal} variants receive attention \cite{doikov2024super}.
Quadratic regularization or Levenberg-Marquardt regularization \cite{levenberg1944method,marquardt1963algorithm} has also been studied for its conceptual simplicity and computational efficiency \cite{li2004regularized,polyak2009regularized}.
Many recent works analyze its iteration complexity, both in the convex \cite{mishchenko2023regularized,doikov2024gradient} and nonconvex setting \cite{gratton2025yet,royer2020newton,he2023newton,zhu2024hybrid,zhou2025regularized}.

\subsection{Properties of the reference function} \label{sec:reference-function}

Throughout this work, we consider valid the following assumptions on the reference function \(\phi\).
\begin{assumption} \label{assumption:reference-basic}
    The following statements hold:
        \begin{assumenum}
            \item The function \(\phi : \R^n \to \exR\) is proper, lsc, \(\mu_\phi\)-strongly convex and even with \(\phi(0) = 0\);
            \item \(\intr \dom \phi \neq \emptyset\); \(\phi \in \cC^2(\intr \dom \phi)\); and for any sequence \(\{x^k\}_{k \in \N}\) that converges to some boundary point of \(\intr \dom \phi\), it follows that \(\Vert \nabla \phi(x^k) \Vert \to \infty\).
        \end{assumenum}
\end{assumption}
\Cref{assumption:reference-basic} ensures in particular that the equivalence \eqref{eq:preconditioned-gradient} holds (cfr.\,\cref{lem:stationary-point-i}), that \(\phi \geq 0\) (cfr.\,\cref{lem:stationary-point-ii}) and that \(\phi^* \in \cC^2\) \cite[p.\,42]{rockafellar1977higher}.
Henceforth, we focus on two types of reference functions:
\begin{enumerate*}[label=(\roman*)]
    \item \emph{isotropic} functions \(\phi(x) = h(\Vert x \Vert)\); and
    \item \emph{separable} functions \(\phi(x) = \sum_{i = 1}^n h(x_i)\)
\end{enumerate*}
where \(h : \R \to \R_+ \cup \{+\infty\}\) is a scalar kernel function.
Interesting choices include
\begin{align*}
    h_1(x) = \cosh(x) &- 1, \quad
    h_2(x) = \exp(\vert x \vert) - \vert x \vert - 1, \quad
    h_3(x) = - \vert x \vert - \ln(1 - \vert x \vert).
\end{align*}
Notably, for \(\phi(x) = \sum_{i = 1}^n h_3(x_i)\) the iterates of \eqref{eq:pg} reduce to a variant of Adam without momentum \cite[Example 1.6]{oikonomidis_nonlinearly_2025}.
We emphasize that \(h_1, h_2, h_3\) grow faster than \(\frac{1}{2} x^2\).
For such reference functions anisotropic smoothness describes more general, and hence more flexible, upper bounds than Lipschitz smoothness. 
Note that the definition in \cite{oikonomidis_nonlinearly_2025} involves two constants \(L, \bar L > 0\).
Without loss of generality \(\bar L\) can be absorbed into \(\phi\) by instead considering \(\widetilde \phi \eqdef \bar L \phi\), see also \cite{oikonomidis2025nonlinearly}.
\begin{definition}[Anisotropic smoothness] \label{def:anisotropic-smoothness}
    A function \(f : \R^n \to \R\) is \(L\)-anisotropically smooth relative to \(\phi\) with constant \(L > 0\) if, for all \(x, \bar x \in \R^n\) and with \(\bar y = \bar x - L^{-1} \nabla \phi^*(\nabla f(\bar x))\),
    \begin{equation} \label{eq:anisotropic-smoothness}
        f(x) \leq f(\bar x) + L^{-1} \phi(L (x - \bar y)) - L^{-1} \phi(L(\bar x - \bar y)).
    \end{equation}
\end{definition}
We now briefly describe the gradient and Hessian of \(\phi^*\) in the isotropic and separable settings.
More details, including the conjugates of \(h_1, h_2, h_3\) and other kernel functions, can be found in \cite{oikonomidis_nonlinearly_2025}. 
\paragraph{Isotropic reference functions} In the isotropic case \(\phi = h \circ \Vert \cdot \Vert\) we have that \(\phi^* = {h^*} \circ \Vert \cdot \Vert\), with
\(
    \nabla \phi^*(y) = {h^*}'(\Vert y \Vert) \barsgn(y)
\)
for all \(y \in \R^n\).
Moreover, the Hessian can be expressed by
\begin{equation*}
    \left[ \nabla^2 \phi^*(y) \right]^{-1} = \frac{1}{{h^*}''(\Vert y \Vert)} \frac{y y^\top}{\Vert y \Vert^2} + \frac{\Vert y \Vert}{{h^*}'(\Vert y \Vert)} \left( \id - \frac{y y^\top}{\Vert y \Vert^2} \right)
\end{equation*}
for \(y \in \R^n \setminus \{0\}\) and else \(\left[ \nabla^2 \phi^*(y) \right]^{-1} = \nicefrac{1}{{h^*}''(\Vert y \Vert)} \id\).

\paragraph{Separable reference functions} In the separable case we have \(\phi^*(y) = \sum_{i = 1}^n {h^*}(y_i) \).
The gradient is of the form \(\nabla \phi^*(y) = \left({h^*}'(y_1), \dots, {h^*}'(y_n)\right)\) and the Hessian is a diagonal matrix with elements
\begin{equation*}
    \left[ \left[ \nabla^2 \phi^*(y) \right]^{-1} \right]_{i,i} = \frac{1}{{h^*}''(y_i)}.
\end{equation*}

%% file: contents/preconditioned_newton.tex
Not any reference function \(\phi\) that satisfies \cref{assumption:reference-basic} is desirable to be used in the context of the preconditioned Newton method described by \eqref{eq:pn}.
We discuss two aspects.

\paragraph{Computational cost of the Newton system} First, the system \eqref{eq:pn} should not be more difficult to solve than the classical Newton system \eqref{eq:n}.
If \(\phi = h \circ \Vert \cdot \Vert\) is isotropic and \(\nabla f(x) \neq 0\), then by \cref{sec:reference-function} we have that
\(
    \nabla^2 f(x) d = - [\nabla^2 \phi^*(\nabla f(x))]^{-1} \nabla \phi^*(\nabla f(x)) = -\alpha(x) \nabla f(x)
\)
where we introduced the shorthand notation
\(
    \alpha(x) := \tfrac{{h^*}'(\Vert \nabla f(x)\Vert)}{{h^*}''(\Vert \nabla f(x)\Vert) \Vert \nabla f(x) \Vert}.
\)
Thus, for isotropic reference functions \eqref{eq:pn} reduces to a conventional Newton method with variable stepsizes \(\alpha(x)\).
\Cref{lem:alpha} provides explicit expressions for \(\alpha(x)\) corresponding to \(h_1, h_2, h_3\).
For these kernel functions \(\alpha(x) \geq 1\) for all \(x \in \R^n\), i.e., preconditioned Newton steps are \emph{larger} than vanilla Newton steps.

On the other hand, if \(\phi(x) = \sum_{i = 1}^n h(x_i)\) is separable, then
\begin{align*}
    \left[ \nabla^2 f(x) d \right]_i &= - \left[ [\nabla^2 \phi^*(\nabla f(x))]^{-1} \nabla \phi^*(\nabla f(x)) \right]_i = - \tfrac{{h^*}'(\nabla_i f(x))}{{h^*}''(\nabla_i f(x))}.
\end{align*}

In both cases, the dominant computational cost remains solving a system with the Hessian $\nabla^2 f$.

\paragraph{Growth of the scalar kernel function} Second, the kernel function \(h\) should grow \emph{sufficiently fast}.

\begin{example} \label{ex:ideal-reference}
    Suppose that \(f\) is strictly convex with minimizer \(x^\star \in \R^n\).
    Then the preconditioned Newton method described by \eqref{eq:pn} with reference function \(\phi(x) = f(x + x^\star)\) converges in a \emph{single iteration}.
    Indeed, we have \(\phi^*(y) = f^*(y) - \langle y, x^\star \rangle\) by \cite[Eq.\,11(3)]{rockafellar_variational_1998}, and hence \(\nabla \phi^*(\nabla f(x^k)) = x^k - x^\star\) and \(\nabla^2 \phi^*(\nabla f(x)) = \left[\nabla^2 f(x)\right]^{-1}\).
    It follows that \(x^{k+1} = x^\star \).
\end{example}
Of course the `ideal' reference function from \cref{ex:ideal-reference} is impractical, as it requires knowledge of the minimizer \(x^\star\).
However, it indicates that the growth of \(\phi\) around \(0\) should ideally \emph{mimic the growth of \(f\) around \(x^\star\)}.
In particular, if \(f\) is not Lipschitz smooth and grows faster than a quadratic, then so should \(\phi\).
This motivates scalar kernel functions like \(h(x) = \cosh(x) - 1 = \sum_{i = 1}^\infty \frac{x^{2i}}{(2i)!}\), which can match arbitrarily fast polynomial growth.

\subsection{Fast local convergence}

Local convergence of Newton's method is often analyzed under Lipschitz continuity of the Hessian \(\nabla^2 f\).
We assume instead Lipschitz continuity of the \emph{preconditioned Hessian}, as in \cite[Assump.\,3.3]{bodard2025escaping}.
\begin{assumption} \label{as:prec-hessian-lipschitz}
    The mapping \(H(x) := \nabla^2 \phi^*(\nabla f(x)) \nabla^2 f(x)\) is \(L_H\)-Lipschitz continuous, i.e.,
    \begin{equation*}
        \exists L_H \geq 0: \Vert H(x) - H(y) \Vert \leq L_H \Vert x - y \Vert, \quad \forall x, y \in \R^n.
    \end{equation*}
\end{assumption}
\Cref{as:prec-hessian-lipschitz} not only subsumes the usual assumption (\(\phi = \frac{1}{2} \Vert \cdot \Vert^2\)), but also holds for costs with stronger growth. 
A proof for the following example is given in \cref{sec:prf-preconditioned-hessian-example}.

\begin{example}[{\cite[Example 3.4]{bodard2025escaping}}] \label{ex:preconditioned-hessian-example}
    Let \(f(x) : \R \to \R\) be a 1D polynomial, and let \(\phi(x) = \cosh(\vert x \vert) - 1\).
    Then there exists an \(L_H > 0\) such that \cref{as:prec-hessian-lipschitz} holds relative to \(\phi\) with constant \(L_H\).
\end{example}
This example reveals that despite a higher-order growth of \(f\), there may exist a \emph{global} constant \(L_H\) for which \cref{as:prec-hessian-lipschitz} holds.
By contrast, a global Hessian Lipschitz constant \(L_{f,2}\) does not exist in the presence of such growth; at best there exists a local constant \(L_{f,2}\) which scales with the diameter of a compact set containing the iterates.
Since these constants determine both the local and global convergence of Newton-based methods, this motivates analyzing \eqref{eq:pn} under \cref{as:prec-hessian-lipschitz}.

If initialized sufficiently close to a strong local minimizer, the iterates of the preconditioned Newton method described by \eqref{eq:pn} converge at least superlinearly.
Note that for isotropic and separable reference functions the Hessian \(\nabla^2 \phi^*(y)\) is a multiple of identity at \(y = 0\).
\begin{theorem} \label{thm:pn-superlinear-quadratic-convergence}
    Let \(x^\star \in \R^n\) be a strong local minimum of \(f \in \cC^2\), i.e., \(\nabla f(x^\star) = 0\) and there exists a \(\mu_f > 0\) such that \(\nabla^2 f(x^\star) \succeq \mu_f I\). Then the following statements hold:
    \begin{theoremenum}
        \item \label{theoremitem:superlinear} there exists a \(\delta > 0\) such that for all \(x^0 \in \cB(x^\star; \delta)\), the iterates \({x^k}_{k \in \N}\) of \eqref{eq:pn} remain in \(\cB(x^\star; \delta)\) and the sequence of iterates converges Q-superlinearly to \(x^\star\);
        \item \label{theoremitem:quadratic} if, additionally, \cref{as:prec-hessian-lipschitz} holds with constant \(L_H \geq 0\), if \(\nabla^2 \phi^*(0) = \widetilde L^{-1} \id\) for some \(\widetilde L > 0\), and if \(\Vert x^0 - x^\star \Vert \leq \Delta \eqdef \frac{2\mu_f}{3\widetilde L L_H}\), then the iterates converge Q-quadratically, i.e.,
        \begin{equation*}
            \forall k \geq 0: \Vert x^{k+1} - x^\star \Vert \leq \frac{\widetilde L L_H}{2(\mu_f - \widetilde L L_H \Vert x^k - x^\star \Vert)} \Vert x^k - x^\star \Vert^2.
        \end{equation*}
    \end{theoremenum}
\end{theorem}
\Cref{thm:pn-superlinear-quadratic-convergence} generalizes \cite[1.2.5]{nesterov_lectures_2018} from the standard (\(\phi = \frac{1}{2} \Vert \cdot \Vert^2\)) to the preconditioned setting, and replaces the conventional Lipschitz constant \(L_{f,2}\) by \(L_H\).

\subsection{Globalization}

Similar to the classical Newton method \eqref{eq:n}, the preconditioned variant \eqref{eq:pn} may diverge if initialized too far from a stationary point.
A standard globalization strategy consists of performing an Armijo linesearch under a Lipschitz-smoothness assumption.
In our framework, the corresponding notion would be anisotropic smoothness (cfr.\,\cref{def:anisotropic-smoothness}).

An Armijo linesearch does not globalize \eqref{eq:pn}, because the preconditioned directions are \emph{no descent directions} in general, i.e, \(\langle \nabla f(x^k), d^k \rangle \leq 0\) may not hold.
Instead, \cref{alg:globalized-pn} defines the candidate point \(x(\tau) \eqdef x^k + \tau d^k + (1 - \tau) p^k\) as a convex combination of a Newton direction \(d^k\) and a negative gradient direction \(p^k\).
As \(\tau \to 0\) we have \(x(\tau) \to p^k\), which yields sufficient decrease under anisotropic smoothness.
A similar strategy is used for globalizing proximal gradient methods \cite{stella2017simple}.
\begin{algorithm}
    \caption{Globalized preconditioned Newton method}
    \label{alg:globalized-pn}
    \begin{algorithmic}[1]
        \REQUIRE{\(x^0 \in \R^n\), \(\alpha \in (0, 1), \sigma \in (0, 1)\)}
        \STATE{Define the stepsize \(\gamma = \nicefrac{\alpha}{L}\)}
        \FOR{$k=0, 1, \dots$}
            \STATE Let \(g^k = \nabla \phi^*(\nabla f(x^k))\); \(p^k = - \gamma g^k\); and \(H_k = \nabla^2 \phi^*(\nabla f(x^k)) \nabla^2 f(x^k)\) 
            \STATE Select \(d^k \in \R^n\) such that \(H_k d^k = - g^k\) if possible, 
            otherwise set \(d^k = -\gamma g^k\)
            \STATE \(x^{k+1} = x^k + \tau_k d^k + (1 - \tau_k) p^k\) where \(\tau_k \in \{\nicefrac{1}{2^i} \mid i \in \N\}\) is the largest stepsize for which \[
                f(x^{k+1}) - f(x^k) \leq -\gamma \sigma \phi(g^k).
            \]
        \ENDFOR
    \end{algorithmic}
\end{algorithm}

If $H^k$ does not have full rank, then there need not exist a $d^k$ satisfying \eqref{eq:pn}.
For simplicity, \cref{alg:globalized-pn} selects a negative gradient direction $d^k = -\gamma g^k$ in such cases. 
\begin{theorem}[Global convergence] \label{thm:pn-globalization}
    Consider the iterates generated by \cref{alg:globalized-pn} and suppose that \(f\) is \(L\)-anisotropically smooth relative to \(\phi\), for some \(L > 0\).
    Then, the linesearch procedure is well-defined, i.e., there exists \(\bar \tau_k > 0\) such that for all \(\tau \in [0, \bar \tau_k]\)
    \begin{equation*}
        f(x^k + \tau d^k + (1 - \tau) p^k) - f(x^k) \leq - \gamma \sigma \phi(\nabla \phi^*(\nabla f(x^k))). 
    \end{equation*}
    Moreover, we obtain a sublinear convergence rate:
    \begin{equation*}
        \min_{0 \leq k \leq K} \phi(\nabla \phi^*(\nabla f(x^k))) \leq \frac{L \left( f(x^0) - \inf f \right)}{\alpha \sigma (K+1)}.
    \end{equation*}
\end{theorem}
Although \(\phi(\nabla \phi^*(\nabla f(x)))\) is a natural stationarity measure under anisotropic smoothness, we note that it can oftentimes be translated to the more standard measure \(\Vert \nabla f(x) \Vert\).
For example, if \(\phi(x) = \cosh(\Vert x \Vert) -1\), then \(\phi(\nabla \phi^*(\nabla f(x))) = \sqrt{1+\Vert \nabla f(x) \Vert^2} - 1\), as derived in \cite[Corollary 3.3]{oikonomidis2025nonlinearly}. 
We also mention that the fast local convergence rate of \eqref{eq:pn} is preserved.
\begin{theorem}[Fast local convergence] \label{thm:globalized-pn-local-convergence}
    Consider the iterates generated by \cref{alg:globalized-pn} and suppose that \(x^k \to x^\star\), with \(x^\star\) a strong local minimum of \(f\). 
    Then, eventually, \(\tau_k = 1\) is always accepted and the iterates satisfy \eqref{eq:pn}.
    In particular, all claims of \cref{thm:pn-superlinear-quadratic-convergence} hold.
\end{theorem}

\begin{remark}[Adaptive estimation of \(L\)] \label{remark:adaptive-L}
    \Cref{alg:globalized-pn} implicitly assumes knowledge of  the constant \(L\) of anisotropic smoothness.
    Yet, this constant can also be estimated \emph{adaptively}, since the convergence proof only requires that 
    \begin{equation*}
        f(x^k - \gamma \nabla \phi^*(\nabla f(x^k))) - f(x^k) \leq -\gamma \phi(\nabla \phi^*(\nabla f(x^k)))
    \end{equation*}
    holds at any iteration.
    Whenever violated, we can set \(L \leftarrow 2 L\) and \(\gamma \leftarrow \nicefrac{\alpha}{L}\), and restart that iteration.
    It is easily verified that this can only happen a finite number of times.
\end{remark}

%% file: contents/cubic-regularization.tex
Practical Newton variants typically include some kind of \emph{regularization} to deal with singular Hessians.
By adding an appropriate positive multiple \(\lambda_k \geq 0\) of the identity matrix, the regularized Hessian \(\nabla^2 f(x^k) + \lambda_k \id\) becomes positive definite.
Regularized Newton variants thus use directions \(s^k\) satisfying
\(
    \left( \nabla^2 f(x^k) + \lambda_k \id \right) s^k = - \nabla f(x^k).
\)
\emph{Cubic regularization} methods select \(\lambda_k = \sigma \Vert s^k \Vert\) for some \(\sigma > 0\).
This direction \(s^k\) can be interpreted as the minimizer of a cubic model of \(f\), hence the name of this type of regularization.
\Cref{alg:pn-cubic-regularization} generalizes this variant to our preconditioning framework and uses directions \(s^k\) satisfying
\begin{equation} \label{eq:cubic-key-relation-2}
    \left( \nabla^2 f(x^k) + \lambda_k M_k \right) s^k = - M_k g^k
\end{equation}
where \(\lambda_k = \sigma \Vert s^k \Vert\) is such that the matrix \(\nabla^2 f(x^k) + \lambda_k M_k\) is positive semidefinite, with
\(
    g^k := \nabla \phi^*(\nabla f(x^k)),
\)
and
\(
    M_k := [\nabla^2 \phi^*(\nabla f(x^k))]^{-1} = \nabla^2 \phi(g^k).
\)
By left-multiplying with \(M_k^{-1}\) \eqref{eq:cubic-key-relation-2} is equivalent to
\begin{equation} \label{eq:cubic-key-relation}
    H(x^k) s^k + \lambda_k s^k = - g(x^k).
\end{equation}
Note that \cref{alg:pn-cubic-regularization} requires $\sigma \geq L_H$ and consequently requires knowledge of the constant $L_H$ from \cref{as:prec-hessian-lipschitz}.
Selecting $\sigma < L_H$ may result in the method not converging.
A variant that adaptively estimates $\sigma$ is presented in \cref{sec:cubic-adaptive}. 
\begin{algorithm}[ht]
    \caption{Preconditioned Newton method with regularization}
    \label{alg:pn-cubic-regularization}
    \begin{algorithmic}[1]
        \REQUIRE{\(x^0 \in \R^n\), \(\sigma \geq L_H\)}
        \FOR{$k=0, 1, \dots$}
            \STATE Let \(g^k = \nabla \phi^*(\nabla f(x^k))\); \(M_k = \nabla^2 \phi(\nabla \phi^*(\nabla f(x^k))) \)
            \STATE Compute a pair \((s^k, \lambda_k) \in \R^n \times \R\) satisfying
            \begin{align}\label{eq:cubic-subproblem}
                M_k^{-1} \left( \nabla^2 f(x^k) + \lambda_k M_k \right) s^k = - g^k, \quad
                \lambda_k = \sigma \Vert s^k \Vert, \quad
                \nabla^2 f(x^k) + \lambda_k M_k \succeq 0
            \end{align}
            \STATE Set \(x^{k+1} = x^k + s^k\).
        \ENDFOR
    \end{algorithmic}
\end{algorithm}

\begin{remark}[No cubic model]
    Although \cref{alg:pn-cubic-regularization} generalizes Newton's method with cubic regularization, it cannot, in general, be interpreted as minimizing a cubic model at each iteration.
    Indeed, define \(F_k(s^k) \eqdef (\nabla^2 f(x^k) + \sigma \Vert s^k \Vert M_k ) s^k\).
    If this is to be the gradient of a (cubic) model, then its Jacobian should be symmetric.
    However, \(J_{F_k}(s^k) =  \nabla^2 f(x^k) + \sigma \Vert s^k \Vert M_k + \sigma \frac{M_k s^k {s^k}^\top}{\Vert s^k \Vert}\) is asymmetric in general.
    A cubic model would exist if \(\lambda_k = \sigma \Vert s^k \Vert_{M_k}\), but then the subsequent convergence proof breaks down as
    \cref{lem:cubic-lem-1,lem:cubic-lem-2} would be expressed in different norms.
\end{remark}

\subsection{Well-definedness of the subproblem}
The following result confirms the well-definedness of the subproblem \eqref{eq:cubic-subproblem}, in the sense that a suitable pair \((s^k, \lambda_k)\) can always be found.
The proof is constructive, and follows a variation on standard arguments for cubic regularization methods \cite[\S 8]{cartis2022evaluation}.
\begin{theorem} \label{thm:cubic-subproblem-well-definedness}
    Let \(A \in \R^{n \times n}\) be symmetric, \(M \in \R^{n \times n}\) symmetric positive definite, \(g \in \R^n \setminus \{0\}\), and \(\sigma > 0\) strictly positive.
    Denote by \(\xi_1 \leq \dots \leq \xi_n\) the generalized eigenvalues of the pencil \((A, M)\), 
    and define
    \(
        s(\lambda) \eqdef -(A + \lambda M)^\dagger M g.
    \)
    Then, there exists a pair \((s^\star, \lambda_\star) \in \R^n \times \R\) satisfying
    \begin{subequations}
        \begin{align}
            (A + \lambda_\star M) s^\star = - M g, \label{eq:cubic-subproblem-eq1}\\ 
            \lambda_\star = \sigma \Vert s^\star \Vert, \label{eq:cubic-subproblem-eq2}\\ 
            A + \lambda_\star M \succeq 0. \label{eq:cubic-subproblem-eq3}
        \end{align}
    \end{subequations}
    In particular, exactly one of the following statements holds:
    \begin{theoremenum}
        \item \(\lambda_\star > -\xi_1\) and \(s^\star = s(\lambda_\star)\); \label{th:subproblem-general-case}
        \item \(\lambda_\star = - \xi_1\), and \(s^\star = s(\lambda_\star) + \alpha_\star v_1\), with \(\alpha_\star \in \R\) such that \(\Vert s^\star \Vert = \nicefrac{\lambda_\star}{\sigma}\), and $v_1$ an eigenvector corresponding to $\xi_1$. \label{th:subproblem-hard-case}
    \end{theoremenum}
\end{theorem}

\subsection{Convergence analysis in the isotropic setting} \label{sec:cubic-pn-convergence}

\begin{algorithm}[ht]
    \caption{Preconditioned Newton method with adaptive regularization}
    \label{alg:pn-adaptive-cubic-regularization}
    \begin{algorithmic}[1]
        \REQUIRE{
            \(x^0 \in \R^n\), \(\sigma_0 > 0\). Provide parameters \(\sigma_{\min} \in (0, \sigma_0], \theta \geq 0\) and \(\eta_1, \eta_2, \eta_3, \gamma_1, \gamma_2, \gamma_3\) satisfying
            \(
                0 < \eta_1 \leq \eta_2 < 1, 0 < \gamma_1 < 1 < \gamma_2 < \gamma_3.
            \)
        }
        \FOR{$k=0, 1, \dots$}
            \STATE Let \(g^k = \nabla \phi^*(\nabla f(x^k))\); \(M_k = \nabla^2 \phi(\nabla \phi^*(\nabla f(x^k))) \)
            \STATE Compute a triplet \((s^k, \lambda_k, z^k) \in \R^n \times \R \times \R^n\) satisfying
            \begin{equation}
                \left\{
                \begin{aligned}
                M_k^{-1} \left( \nabla^2 f(x^k) + \lambda_k M_k \right) s^k &= - g^k + z^k,\\
                \lambda_k &= \sigma_k \Vert s^k \Vert,\\
                \nabla^2 f(x^k) + \lambda_k M_k &\succeq 0,\\
                \Vert z^k \Vert &\leq \frac{1}{2} \theta \Vert s^k \Vert^2.
                \end{aligned}
                \right.
                \end{equation}
            \STATE Compute \(f(x^k + s^k)\) and define
            \(
                \rho_k = \frac{f(x^k) - f(x^k+s^k)}{(\nicefrac{\widetilde L}{4}) \sigma_k \Vert s^k \Vert^3}.
            \)
            \STATE{ \textbf{If } \(\rho_k \geq \eta_1\) \textbf{then}} \(x^{k+1} = x^k + s^k\) \textbf{else} \(x^{k+1} = x^k\)
            \STATE Set
            \begin{equation} \label{eq:sigma-update}
                \sigma_{k+1} \in \begin{cases}
                    [\max(\sigma_{\min}, \gamma_1 \sigma_k), \sigma_k] & \rho_k \geq \eta_2,\\
                    [\sigma_k, \gamma_2 \sigma_k] & \rho_k \in [\eta_1, \eta_2),\\
                    [\gamma_2 \sigma_k, \gamma_3 \sigma_k] & \rho_k < \eta_1.
                \end{cases}
            \end{equation}
        \ENDFOR
    \end{algorithmic}
\end{algorithm}

We now analyze the convergence of \cref{alg:pn-cubic-regularization} for \emph{isotropic} reference functions.
\begin{assumption}\label{ass:isotropic-reference}
    The reference \(\phi(x) = h(\Vert x \Vert)\) is isotropic, and 
    the function \(\frac{y}{{h^*}'(y)}\) is continuous and uniformly lower bounded by \(\widetilde L > 0\) on \(\R_+\).
\end{assumption}
Note that \(z \mapsto \frac{\Vert z \Vert}{{h^*}'(\Vert z \Vert)} \geq 1\) for \(h_1, h_2, h_3\), with equality holding only at \(z = 0\).
Under \cref{ass:isotropic-reference}, we therefore have that \(\nabla f(x^k) = \nu(x^k) \nabla \phi^*(\nabla f(x^k))\) where
\(
    \nu(x) := \frac{\Vert \nabla f(x) \Vert}{{h^*}'(\Vert \nabla f(x)\Vert)}
\)
is continuous and uniformly lower bounded, i.e., \(\nu(x) \geq \widetilde L\).
Unlike for generic preconditioners, \cref{ass:isotropic-reference} ensures that \(s^k\) is a descent direction.
Indeed, observe from \eqref{eq:cubic-key-relation-2} that
\(
    \left( \nabla^2 f(x^k) + \lambda_k M_k \right) s^k = - M_k g(x^k) = -\alpha(x^k) \nabla f(x^k).
\)
By positive semidefiniteness of \(\nabla^2 f(x^k) + \lambda_k M_k\), the update directions \(s^k\) therefore satisfy
\begin{equation} \label{eq:s-descent}
    \langle \nabla f(x^k), s^k \rangle = -\tfrac{1}{\alpha(x^k)} \langle s^k, \left( \nabla^2 f(x^k) + \lambda_k M_k \right) s^k \rangle \leq 0.
\end{equation}
Traditional analyses based on $L_{f,2}$-Lipschitz continuous Hessians exploit upper bounds of the form
\begin{equation*}
    \forall x, s \in \bR^n: \left| f(x+s) - f(x) - \nabla f(x)^T s - \frac{1}{2} s^T \nabla^2 f(x)s \right| \leq \frac{L}{6}\|s\|^3.
\end{equation*}
A similar result is nontrivial to obtain, and in fact does not hold under \cref{as:prec-hessian-lipschitz} for generic reference functions \(\phi\).
Nevertheless, the following key lemma establishes sufficient decrease by exploiting properties of the directions $s^k$ generated by \cref{alg:pn-cubic-regularization}.
\begin{lemma} \label{lem:cubic-lem-1}
    Consider \(x^k, s^k \in \R^n\) generated by \cref{alg:pn-cubic-regularization}.
    If \cref{as:prec-hessian-lipschitz,ass:isotropic-reference} hold, then
    \begin{equation*}
        f(x^k+s^k) - f(x^k) \leq - \frac{\widetilde L L_H}{3} \Vert s^k \Vert^3.
    \end{equation*}
\end{lemma}
Our analysis also requires a bound on \(\Vert s^k \Vert\) in terms of the (preconditioned) gradient norm.
\begin{lemma} \label{lem:cubic-lem-2}
    Consider \(x^k, s^k \in \R^n\) generated by \cref{alg:pn-cubic-regularization}.
    If \cref{as:prec-hessian-lipschitz} holds, then
    \begin{equation*}
        \Vert s^k \Vert^2 \geq \frac{2}{2\sigma + L_H} \Vert g(x^{k+1}) \Vert.
    \end{equation*}
\end{lemma}
The combination of \cref{lem:cubic-lem-1,lem:cubic-lem-2} and a telescoping argument then yields the following complexity result.
\begin{theorem} \label{thm:cubic-pn-convergence-rate}
    Consider \(x^k, s^k \in \R^n\) generated by \cref{alg:pn-cubic-regularization}.
    If \cref{as:prec-hessian-lipschitz,ass:isotropic-reference} hold and if \(\sigma \geq L_H\), then \cref{alg:pn-cubic-regularization} finds a point such that \(\Vert \nabla \phi^*(\nabla f(x^K)) \Vert \leq \varepsilon\) in at most $K$ iterations, where
    \begin{equation}\label{eq:cubic-pn-convergence-rate}
        K  \leq \frac{6 \sigma^{\nicefrac{3}{2}} (f(x^0) - \inf f)}{\widetilde L L_H \varepsilon^{\nicefrac{3}{2}}}.
    \end{equation}
\end{theorem}
If \(\sigma = L_H\), then \eqref{eq:cubic-pn-convergence-rate} simplifies to
\(
    K \leq \frac{6 \sqrt{L_H} (f(x^0) - \inf f)}{\widetilde L \varepsilon^{\nicefrac{3}{2}}}.
\)
Observe that the preconditioned gradient norm \(\Vert \nabla \phi^*(\nabla f(x^K)) \Vert\) is used as optimality measure in \cref{thm:cubic-pn-convergence-rate}, instead of the standard gradient norm \(\Vert \nabla f(x^K) \Vert\).
Oftentimes, both can be related. 
For example, suppose that \(\phi(x) = \widetilde L h_3(\Vert x \Vert)\).
Then \(
    \nabla \phi^*(y) = \frac{y}{\widetilde L + \Vert y \Vert}.
\)
Since \(\Vert \nabla \phi^*(\nabla f(x^K)) \Vert \geq \frac{\Vert \nabla f(x^K) \Vert}{\widetilde L + G} \) where \(G = \max_{0\leq k \leq K} \Vert \nabla f(x^k) \Vert\).
In terms of the usual gradient norm, this yields a complexity $\cO(\sqrt{\widetilde L L_H} \varepsilon^{-\nicefrac{3}{2}})$, similar to \cite{zhou2025regularized}.

\subsection{Adaptive regularization} \label{sec:cubic-adaptive}

Finally, \cref{alg:pn-adaptive-cubic-regularization} describes an \emph{adaptive} variant of \cref{alg:pn-cubic-regularization}.
An obvious advantage is that the Lipschitz constant \(L_H\) is not explicitly required. 
Moreover, it supports inexact solutions to the subproblems, represented by the vector \(z^k \in \R^n\). Clearly, if \(z^k = 0\) (e.g.\,when setting \(\theta = 0\)), then the obtained subproblem solution is exact and corresponds to that of \cref{alg:pn-cubic-regularization}.
Since \(z^k = 0\) can always be chosen, well-definedness immediately follows.

\begin{theorem} \label{thm:adaptive-cubic-pn-convergence-rate}
    If \cref{as:prec-hessian-lipschitz,ass:isotropic-reference} hold, then \cref{alg:pn-adaptive-cubic-regularization} requires at most
    \[
        K \leq C_1 \varepsilon^{-\nicefrac{3}{2}} + C_2
    \]
    iterations, with \(C_1, C_2 > 0\) independent of \(\varepsilon\), to find a point such that \(\Vert \nabla \phi^*(\nabla f(x^K)) \Vert \leq \varepsilon\).
\end{theorem}

%% file: contents/numerics.tex
\begin{figure*}[t]
    \centering
    \includegraphics[width=.99\linewidth]{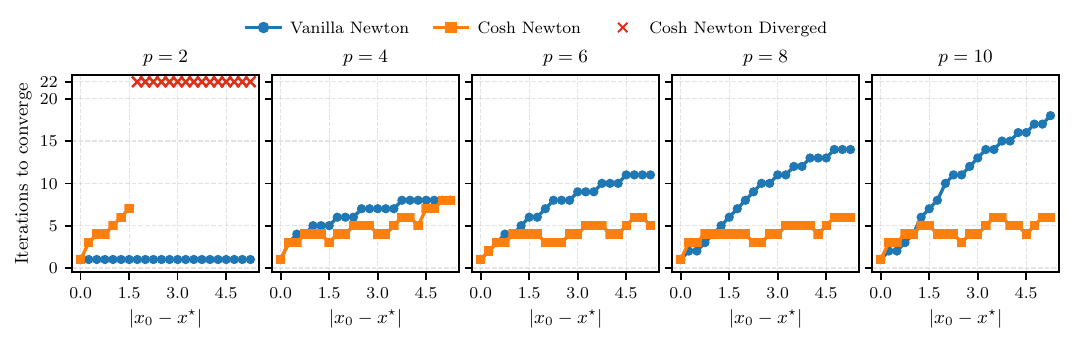}
    \caption{Iterations to minimize \(f(x) = \frac{1}{p} x^p + \frac{1}{2} x^2\): \eqref{eq:pn} is consistent for different \(p\) and \(x^0\).}
    \label{fig:1d-polynomial-iters}
\end{figure*}

\subsection{Higher-order 1D polynomials}

This experiment illustrates the importance of accurate smoothness conditions.
Consider a scalar polynomial \(f(x) = \frac{1}{p} x^p + \frac{1}{2} x^2 \) where \(p \geq 2\) is even. 
The unique minimizer is clearly \(x^\star = 0\). 
\Cref{fig:1d-polynomial-iters} compares vanilla Newton \eqref{eq:n} against preconditioned Newton \eqref{eq:pn} with \(\phi(x) = \cosh(\vert x \vert) - 1\), and this for various initializations \(x^0\).
No globalization is applied.
Vanilla Newton obviously performs best for \(p = 2\), as it converges in a single step.
If \(p > 2\), then preconditioned Newton performs remarkably better.
In fact, \eqref{eq:n} quickly deteriorates as \(p\) or \(x^0\) is increased, whereas \eqref{eq:pn} is consistent.

\subsection{Logistic regression}
\begin{figure*}[t]
    \centering

    \begin{minipage}[t]{0.37\linewidth}
        \centering
        \includegraphics[width=\linewidth]{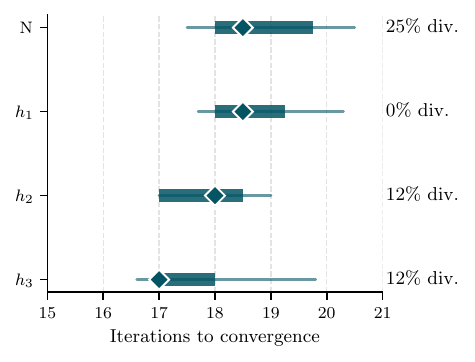}
        \caption{\eqref{eq:n} and \eqref{eq:pn}, without globalization, on logistic regression problems from the LIBSVM dataset \cite{chang2011libsvm}.} %
        \label{fig:logreg-no-globalization}
    \end{minipage}
    \hfill
    \begin{minipage}[t]{0.6\linewidth}
        \centering
        \includegraphics[width=\linewidth]{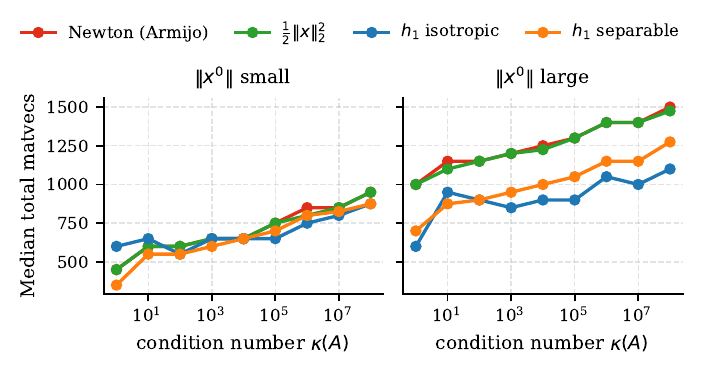}
        \caption{Vanilla versus preconditioned Newton, with globalization, for synthetic symmetric matrix factorization problems. Medians over $80$ runs are reported.}
        \label{fig:symmf_globalization}
    \end{minipage}

\end{figure*}

We evaluate the proposed methods on logistic regression problems from the LIBSVM dataset \cite{chang2011libsvm}. 
A run is successful if the gradient norm drops below \(10^{-6}\) within at most \(100\) iterations.
\Cref{fig:logreg-no-globalization} visualizes the number of iterations of \eqref{eq:n} and \eqref{eq:pn} until convergence for the problems \texttt{w1a-w8a} when initialized in the zero vector.
For \eqref{eq:pn}, we use $\phi = \bar L h_i(\Vert \cdot \Vert)$ with $i \in \{1, 2, 3\}$, and $\bar L$ heuristically chosen as the initial gradient norm.
The preconditioned variants clearly require fewer iterations.
We emphasize that at every iteration, all methods solve a linear system with the same Hessian but a different right-hand side, as noted in \cref{sec:pn}.
This yields a similar computational cost per iteration.
The preconditioned variants also diverge less, although, in our experience, this depends on the scaling of $\bar L$.

\subsection{Matrix factorization}

Finally, we validate the preconditioned Newton methods on symmetric matrix factorization.
Data matrices are synthetically generated with condition numbers ranging from \(1\) to \(10^8\) and the variables are matrices in \(\R^{10 \times 5}\).
\Cref{fig:symmf_globalization} compares the globalized variant \cref{alg:globalized-pn} against vanilla Newton with Armijo linesearch.
For each condition number we report the median number of matrix-vector products.
Our implementation uses a standard conjugate gradients (CG) method for solving \eqref{eq:n}.
The preconditioned systems \eqref{eq:pn} involve asymmetric matrices, and are solved using a standard generalized minimal residual (GMRES) method.
If the initial iterate is relatively small -- in this experiment it is sampled elementwise from a standard normal Gaussian distribution -- then all preconditioners perform quite similar.
However, if the magnitude of the initial iterate is enlarged by a factor \(100\), we observe that the preconditioned variants significantly outperform vanilla Newton. 
It is precisely in this setting where an accurate global smoothness constant is valuable. 
Also observe that all methods scale similarly with the condition number of the problem.

We refer to \cref{sec:additional-experiments} for additional experiments involving \cref{alg:pn-cubic-regularization,alg:pn-adaptive-cubic-regularization}.

%% file: contents/conclusion.tex
In this work, we introduced and studied a novel class of Newton methods that is derived by applying a nonlinear preconditioner to the first-order optimality conditions. Our approach is based on generalized smoothness notions like anisotropic smoothness, which may better describe an objective in the presence of higher-order polynomial growth. 
We analyzed both the local and global convergence of our scheme under such relaxed smoothness conditions, and discussed practical choices for the preconditioners.
We also incorporated a variant of cubic regularization into our approach, and established an \(\cO(\varepsilon^{-\nicefrac{3}{2}})\) iteration complexity.
Finally, we qualitatively validated our theoretical results on logistic regression and matrix factorization problems.

Directions for future work include (i) extending of our framework to quadratic regularization or even trust-region variants; (ii) applying preconditioned Newton techniques to challenging problems like training Physics-Informed Neural Networks \cite{rathore2024challenges}; and (iii) designing novel classes of preconditioners with desirable properties.

%% file: appendix.tex
\section{Additional Results}

\begin{lemma}[Stationarity] \label{lem:stationary-point}
    Suppose that \cref{assumption:reference-basic} holds.
    Then,
    \begin{lemenum}
        \item \(x^\star \in \R^n\) is a stationary point of \(f\), i.e., \(\nabla f(x^\star) = 0\), if and only if \(\nabla \phi^*(\nabla f(x^\star)) = 0\) \label{lem:stationary-point-i}
        \item \(\phi(z) \geq 0\) for all \(z \in \R^n\). \label{lem:stationary-point-ii}
    \end{lemenum}
\end{lemma}
\begin{proof}
    Strong convexity of \(\phi\) implies that $\phi$ has a unique stationary point $z^\star$ where $\nabla \phi(z^\star) = 0$. 
    Since $\phi$ is even, $\phi(z^\star) = \phi(-z^\star)$, meaning $-z^\star$ is also a minimizer. 
    By uniqueness, $z^\star = -z^\star$, which implies $z^\star = 0$. Thus, $\nabla \phi(0) = 0$.
    From the Fenchel-Moreau identity $\nabla \phi(z) = y \iff \nabla \phi^*(y) = z$, it follows that $\nabla \phi^*(y) = 0$ if and only if \(y = 0\), which proves the first claim.

    The second claim then follows from the fact that \(\phi(0) = 0\) and from the fact that \(0\) is the minimizer of \(\phi\).
\end{proof}

\begin{lemma} \label{lem:alpha}
    Let \(\phi = h \circ \Vert \cdot \Vert\) be isotropic. Then, \(\alpha(x) := \tfrac{{h^*}'(\Vert \nabla f(x)\Vert)}{{h^*}''(\Vert \nabla f(x)\Vert) \Vert \nabla f(x) \Vert}\) can be expressed as follows.
    \begin{lemenum}
        \item If \(h(x) = \cosh(x) - 1\), then
            \(
                \alpha(x) = \arcsinh(\Vert \nabla f(x) \Vert) \frac{\sqrt{1+\Vert \nabla f(x)\Vert^2}}{\Vert \nabla f(x) \Vert}.
            \)
        \item If \(h(x) = \exp(\vert x \vert) - \vert x \vert - 1\), then
            \(
                \alpha(x) = \ln(1+\Vert \nabla f(x) \Vert) \frac{1+ \Vert \nabla f(x) \Vert}{\Vert \nabla f(x) \Vert}.
            \)
        \item If \(h(x) = - \vert x \vert - \ln(1 - \vert x \vert)\), then 
            \(
                \alpha(x) = 1 + \Vert \nabla f(x) \Vert.
            \)
    \end{lemenum}
    Moreover, \(\alpha(x) \geq 1\) for these three kernel functions.
\end{lemma}
\begin{proof}
    The first claim follows immediately by substituting into the definition of $\alpha$ the expressions for \({h^*}'\) and \({h^*}''\), which can be found in \cite[Table 1]{oikonomidis_nonlinearly_2025}.

    As for the second claim, standard calculus reveals that for any \(y \geq 0\) 
    \begin{equation*}
        \arcsinh(y) \frac{\sqrt{1+y^2}}{y} \geq 1, \qquad \ln(1+y) \frac{1+ y}{y} \geq 1, \qquad 1 + y \geq 1,
    \end{equation*}
    which proves the claim.
\end{proof}

\subsection{Proof of \cref{ex:preconditioned-hessian-example}} \label{sec:prf-preconditioned-hessian-example}

\begin{proof}
    Straightforward manipulations reveal that
    \(
        H(x)=f''(x)(1+(f'(x))^2)^{-\nicefrac{1}{2}}.
    \)
    It suffices to uniformly upper bound the magnitude of \(H'(x)\), which is of the form
    \begin{equation*}
        H'(x) = \frac{f'''(x)}{\sqrt{1+(f'(x))^2}} - \frac{(f''(x))^2 f'(x)}{(1+(f'(x))^2)^{\nicefrac{3}{2}}}.
    \end{equation*}
    Since \(H'(x)\) is continuous on \(\R\) (because the denominators are always \(\geq 1\)), a sufficient condition for boundedness of \(H'(x)\) is that \(\lim_{x \to \pm \infty} H'(x) = 0\). 
    Without loss of generality, assume that \(f(x) = \sum_{i}^{d} \alpha_i x^i\), where \(\alpha_d \neq 0\).
    Then,
    \begin{align*}
        \lim_{x \to \pm \infty} H'(x) &= \lim_{x \to \pm \infty} \frac{f'''(x)}{\sqrt{(f'(x))^2}} - \frac{(f''(x))^2 f'(x)}{((f'(x))^2)^{\nicefrac{3}{2}}}\\
        &= \lim_{x \to \pm \infty} \frac{\alpha_d x^{d-3}}{\sqrt{\alpha_d^2 (x^{d-1})^2}} - \frac{(\alpha_d x^{d-2})^2 \alpha_d x^{d-1}}{((\alpha_d x^{d-1})^2)^{\nicefrac{3}{2}}}\\
        &= \lim_{x \to \pm \infty} \frac{1}{x^2} - \lim_{x \to \pm \infty} \frac{1}{x^2} = 0 - 0 = 0.
    \end{align*}
    This concludes the proof.
\end{proof}

\section{Proofs of \cref{sec:pn}}

\input{appendix/pn-superlinear-quadratic-convergence.tex}

\input{appendix/pn-globalization.tex}

\input{appendix/globalized-pn-local-convergence.tex}

\section{Proofs of \cref{sec:cubic-pn}}
\input{appendix/cubic-subproblem-well-definedness.tex}

\input{appendix/cubic-lem-1.tex}

\input{appendix/cubic-lem-2.tex}

\input{appendix/cubic-pn-convergence-rate.tex}

\input{appendix/adaptive-cubic-pn-convergence-rate.tex}

\section{Additional experiments} \label{sec:additional-experiments}

This section provides some additional matrix factorization experiments. 
Our setup is similar to before, but now we focus on \cref{alg:pn-cubic-regularization,alg:pn-adaptive-cubic-regularization}.
First, \cref{fig:mf-cubic} illustrates the benefits of better smoothness constants in the context of \cref{alg:pn-cubic-regularization}.
Recall that the algorithm requires \(\sigma \geq L_H\), and likewise cubic regularization methods require \(\sigma \geq L_{f, 2}\).
Of course we do not know these constants in practice, so \cref{fig:mf-cubic} visualizes two different values, i.e., \(\sigma = 1\) and \(\sigma = 5\).
For \(\sigma = 5\) both methods perform more or less similar.
However, for \(\sigma = 1\) only \cref{alg:pn-cubic-regularization} converges.
This suggests that \(L_{f,2}\) is larger than \(L_H\) here.
We also observe that a larger value of \(\sigma\) requires more iterations, as there is more regularization.
\begin{figure}[ht]
    \centering
    \includegraphics[width=.6\linewidth]{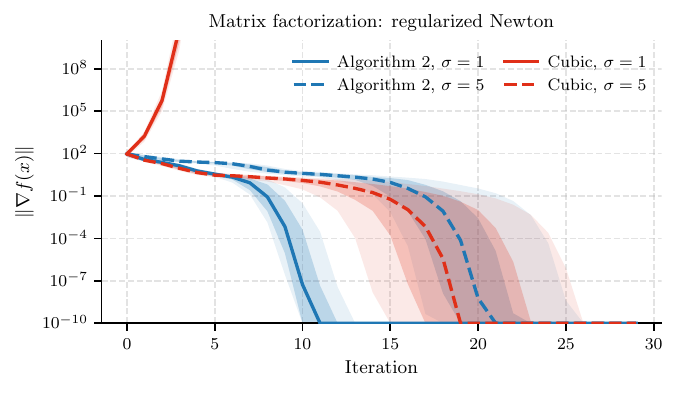}
    \caption{Gradient norm versus function iterations for 80 random matrix factorization problems. Vanilla and preconditioned variants of cubic regularization are compared.}
    \label{fig:mf-cubic}
\end{figure}
\Cref{fig:mf-adaptive-cubic} compares \cref{alg:pn-adaptive-cubic-regularization} and the adaptive cubic regularization method.
As \(\sigma_k\) is now dynamically adjusted, the difference between preconditioning and not preconditioning becomes smaller.
\begin{figure}[ht]
    \centering
    \includegraphics[width=.6\linewidth]{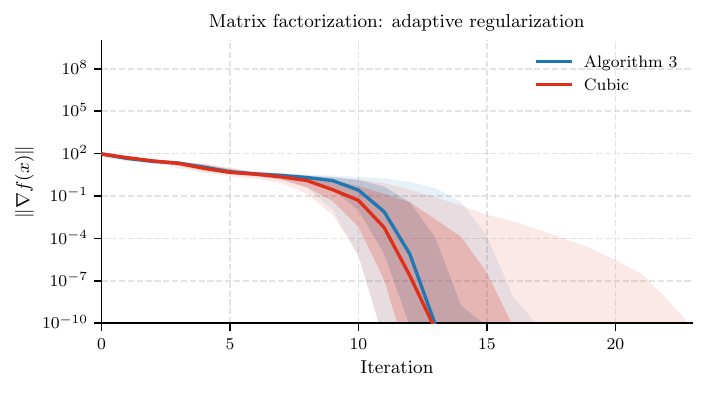}
    \caption{Gradient norm versus function iterations for 80 random matrix factorization problems. Vanilla and preconditioned variants of adaptive cubic regularization are compared.}
    \label{fig:mf-adaptive-cubic}
\end{figure}

%% file: appendix/pn-superlinear-quadratic-convergence.tex
\subsection{Proof of \cref{thm:pn-superlinear-quadratic-convergence}}

\begin{proof}
Let $g(x) = \nabla \phi^*(\nabla f(x))$ with Jacobian matrix $H(x)= \nabla^2 \phi^*(\nabla f(x)) \nabla^2 f(x)$. The iteration \eqref{eq:pn} can then be written as
\(
    x^{k+1} = x^k - H(x^k)^{-1} g(x^k).
\)

``\Cref{theoremitem:superlinear}'': 
At $x^\star$, we have $H(x^\star) = \nabla^2 \phi^*(\nabla f(x^\star)) \nabla^2 f(x^\star) = \nabla^2 \phi^*(0) \nabla^2 f(x^\star)$. Since $\phi$ is strongly convex, we have $\nabla^2 \phi(0) \succ 0$. By properties of the convex conjugate, $\nabla^2 \phi^*(0) = (\nabla^2 \phi(0))^{-1} \succ 0$. As $\nabla^2 f(x^\star) \succ \mu_f I$, we observe that $H(x^\star)$ is the product of two positive definite matrices and is thus invertible.
By continuity of $H(x)$, there exists a \(\delta > 0\) and $\mu > 0$ such that for all $x \in \mathcal{B}(x^\star; \delta)$, the preconditioned Hessian $H(x)$ is invertible and $\Vert H(x)^{-1} \Vert \leq 1/\mu$.

Taylor's theorem with integral remainder states that for $g(x)$ we have
\[
    g(y) = g(x) + H(x) (y - x) + R(x, y),
\]
with the residual term given by
\(
    R(x, y) = \int_{0}^{1} \left[ H(x + t (y - x)) - H(x) \right] (y - x) \mathrm d t.
\)
Using the iteration map $x^{k+1} = x^k - H(x^k)^{-1} g(x^k)$, we can write the error $x^{k+1} - x^\star$ as
\begin{equation*}
    \begin{aligned}
        x^{k+1} - x^\star &= x^{k} - H(x^k)^{-1} g(x^k) - x^\star \\
        &= (x^k - x^\star) - H(x^k)^{-1} g(x^k) \\ 
        &= H(x^k)^{-1} \left( H(x^k) (x^k - x^\star) - g(x^k) \right)\\
        &= H(x^k)^{-1} \left( H(x^k) (x^k - x^\star) - (g(x^k) - g(x^\star)) \right)\\
        &= H(x^k)^{-1} \left( g(x^\star) - g(x^k) - H(x^k) (x^\star - x^k) \right)\\
        &= H(x^k)^{-1} R(x^k, x^\star).
    \end{aligned}
\end{equation*}
Here the fourth equality uses $g(x^\star) = 0$ (cfr.\cref{lem:stationary-point}).
Note that for all $x \in \R^n$ we have 
\begin{equation*}
    \begin{aligned}
        \Vert R(x, x^\star) \Vert &\leq \int_{0}^{1} \Vert [H(x + t (x^\star - x)) - H(x)] (x^\star - x) \Vert \mathrm d t\\
        &\leq \Vert x^\star - x \Vert \int_{0}^{1} \Vert H(x + t (x^\star - x)) - H(x) \Vert \mathrm d t
    \end{aligned}
\end{equation*}
and by continuity of $H$ it moreover holds that
\(
    \lim_{x \to x^\star} \int_{0}^{1} \Vert H(x + t (x^\star - x)) - H(x) \Vert \mathrm d t = 0.
\)

Thus, we can shrink $\delta > 0$ (if necessary) such that for all $x \in \mathcal{B}(x^\star; \delta)$ it holds that
\begin{equation*}
    \int_{0}^{1} \Vert H(x + t (x^\star - x)) - H(x) \Vert \mathrm d t \leq \frac{\mu}{2},
\end{equation*}
which implies $\Vert R(x, x^\star) \Vert \leq \frac{\mu}{2} \Vert x - x^\star \Vert$.
Using this bound together with the fact that $\Vert H^{-1}(x^k) \Vert \leq 1/\mu$ for all $x^k \in \mathcal{B}(x^\star; \delta)$, we obtain
\begin{equation*}
    \Vert x^{k+1} - x^\star \Vert \leq \Vert H^{-1}(x^k) \Vert \Vert R(x^k, x^\star) \Vert \leq \frac{1}{\mu} \left( \frac{\mu}{2} \Vert x^k - x^\star \Vert \right) = \frac{1}{2} \Vert x^k - x^\star \Vert,
\end{equation*}
which shows by induction that the iterates remain in $\mathcal{B}(x^\star; \delta)$ and converge (at least) Q-linearly to $x^\star$.
The Q-superlinear convergence follows from the fact that
\begin{equation*}
    \lim_{k \to \infty} \frac{\Vert x^{k+1} - x^\star \Vert}{\Vert x^k - x^\star \Vert} \leq \lim_{k \to \infty} \Vert H^{-1}(x^k) \Vert \int_{0}^{1} \Vert H(x^k + t (x^\star - x^k)) - H(x^k) \Vert \mathrm d t = 0.
\end{equation*}
Indeed, as $k \to \infty$, we have $x^k \to x^\star$. The first term satisfies $\Vert H^{-1}(x^k) \Vert \leq 1/\mu$, and the second term (the integral) goes to $0$ by continuity of \(H\).

``\Cref{theoremitem:quadratic}'': 
Since $H$ is $L_H$-Lipschitz continuous, we have
\[
    H(x^k) = H(x^\star) + E(x^k), \qquad E(x^k) := H(x^k) - H(x^\star), \qquad \text{where } \Vert E(x^k)\Vert \leq L_H \Vert x^k - x^\star\Vert.
\]
We remark that \(\nabla^2 \phi^*(0) = \widetilde L^{-1} \id\) is a multiple of identity, and hence $H(x^\star)$ is symmetric positive definite and satisfies
\(
\lambda_{\min}(H(x^\star)) \geq \frac{\mu_f}{\widetilde L} > 0.
\)
For any $v \in \mathbb{R}^n$ with $\Vert v \Vert=1$, we have
\[
    \Vert H(x^k)v\Vert = \Vert (H(x^\star)+E(x^k))v\Vert \geq \Vert H(x^\star)v\Vert  - \Vert E(x^k)v\Vert \geq \tfrac{\mu_f}{\widetilde L} - L_H \Vert x^k - x^\star \Vert.
\]
Taking the infimum over all unit vectors yields
\[
\sigma_{\min}(H(x^k)) = \inf_{\Vert v\Vert =1} \Vert H(x^k)v\Vert \geq \frac{\mu_f}{\widetilde L} - L_H \Vert x^k - x^\star\Vert .
\]
Consequently, if $\Vert x^k - x^\star\Vert  < \frac{\mu_f}{\widetilde L L_H}$, then $H(x^k)$ is invertible and
\[
\Vert H(x^k)^{-1}\Vert = \sigma_{\min}(H(x^k))^{-1} \leq \bigl(\frac{\mu_f}{\widetilde L} - L_H \Vert x^k - x^\star\Vert \bigr)^{-1} = \frac{\widetilde L}{\mu_f - \widetilde L L_H \Vert x^k - x^\star \Vert}.
\]
\Cref{as:prec-hessian-lipschitz} also implies that
\(
    \Vert H(x + t (x^\star - x)) - H(x) \Vert \leq L_H \Vert t(x^\star - x) \Vert = L_H t \Vert x - x^\star \Vert.
\)
Plugging this into the residual bound yields
\begin{equation*}
     \Vert R(x^k, x^\star) \Vert \leq \Vert x^k - x^\star \Vert \int_{0}^{1} (L_H t \Vert x^k - x^\star \Vert) \mathrm d t = L_H \Vert x^k - x^\star \Vert^2 \int_{0}^{1} t \mathrm d t = \frac{L_H}{2} \Vert x^k - x^\star \Vert^2.
\end{equation*}
The desired quadratic rate now follows:
\[
    \Vert x^{k+1} - x^\star \Vert \leq \Vert H^{-1}(x^k) \Vert \Vert R(x^k, x^\star) \Vert \leq \frac{\widetilde L L_H}{2(\mu_f - \widetilde L L_H \Vert x^k - x^\star \Vert)} \Vert x^k - x^\star \Vert^2.
\]
\end{proof}

%% file: appendix/pn-globalization.tex
\subsection{Proof of \cref{thm:pn-globalization}}

\begin{proof}
    We first proof the well-definedness of the linesearch procedure. Then, we establish the claimed convergence rate.

    \paragraph{Well-definedness of the linesearch procedure} Let us fix an arbitrary iterate \(k \in \bN\). We omit sub- and superscripts for notational simplicty.
    Suppose, by contradiction, that for all \(\varepsilon > 0\) there exists \(\tau_\varepsilon \in [0, \varepsilon]\) such that the point \(x_\varepsilon \eqdef x + \tau_\varepsilon d + (1 - \tau_\varepsilon) p\) satisfies
    \begin{equation*}
        f(x_\varepsilon) - f(x) > -\gamma \sigma \phi(\nabla \phi^*(\nabla f(x))).
    \end{equation*}
    By taking the limit \(\varepsilon \downto 0\) it follows by (strict) continuity of \(f\) that 
    \begin{equation} \label{prf:pn-globalization-eq-1}
        f(x + p) - f(x) \geq - \gamma \sigma \phi(\nabla \phi^*(\nabla f(x))).
    \end{equation}
    On the other hand, the anisotropic smoothness inequality \eqref{eq:anisotropic-smoothness} evaluated at the points \(x + p\) and \(x\) yields
    \begin{equation*}
        f(x + p) - f(x) \leq L^{-1} \left[ \phi((1-\alpha) \nabla \phi^*(\nabla f(x))) - \phi(\nabla \phi^*(\nabla f(x))) \right].
    \end{equation*}
    By convexity of \(\phi\) we have that for any \(\theta \in [0, 1]\)
    \begin{equation} \label{prf:pn-globalization-eq-2}
        \phi(\theta x) = \phi((1-\theta) 0 + \theta x) \leq (1-\theta) \phi(0) + \theta \phi(x) = \theta \phi(x).
    \end{equation}
    Note that \(\alpha \in (0, 1)\) and $\gamma = \nicefrac{\alpha}{L}$, such that we obtain
    \begin{equation*}
        f(x + p) - f(x) \leq - \gamma \phi(\nabla \phi^*(\nabla f(x))).
    \end{equation*}
    Since \(\sigma \in (0, 1)\), the combination of \eqref{prf:pn-globalization-eq-1} and \eqref{prf:pn-globalization-eq-2} yields a contradiction, i.e., \(f(x + p) - f(x) > f(x + p) - f(x)\), which proves the well-definedness of the linesearch procedure.

    \paragraph{Convergence rate} The linesearch condition guarantees that for all \(k \geq 0\)
    \begin{equation*}
        f(x^{k+1}) - f(x^k) \leq -\gamma \sigma \phi(\nabla \phi^*(\nabla f(x^k))).
    \end{equation*}
    A standard telescoping argument yields
    \begin{equation*}
        \sum_{k = 0}^K \phi(\nabla \phi^*(\nabla f(x^k))) \leq \frac{f(x^0) - f(x^K)}{\gamma \sigma} \leq \frac{f(x^0) - \inf f}{\gamma \sigma}.
    \end{equation*}
    Lower bounding \(\sum_{k = 0}^K \phi(\nabla \phi^*(\nabla f(x^k))) \geq (K+1) \min_{0 \leq k \leq K} \phi(\nabla \phi^*(\nabla f(x^k)))\) and substituting $\gamma = \nicefrac{\alpha}{L}$ establishes the claim.
\end{proof}

%% file: appendix/globalized-pn-local-convergence.tex
\subsection{Proof of \cref{thm:globalized-pn-local-convergence}}

\begin{proof}
    By \cref{thm:pn-superlinear-quadratic-convergence} the update directions \(d^k\) are eventually Q-superlinear, meaning that \[
        \lim_{k \to \infty} \frac{\Vert x^k + d^k - x^\star \Vert}{\Vert x^k - x^\star \Vert} = 0.
    \]
    Hence, it follows that 
    \begin{equation*}
        \begin{aligned}
            \varepsilon_k &:= \frac{f(x^k + d^k) - f(x^\star)}{f(x^k) - f(x^\star)}\\
            &= \frac{\frac{1}{2} \langle \nabla^2 f(x^\star) (x^k + d^k - x^\star), x^k + d^k - x^\star\rangle + o(\Vert x^k + d^k - x^\star \Vert^2)}{\frac{1}{2} \langle \nabla^2 f(x^\star) (x^k - x^\star), x^k - x^\star\rangle + o(\Vert x^k - x^\star \Vert^2)}\\
            &\leq \frac{\Vert \nabla^2 f(x^\star) \Vert \left(\frac{\Vert x^k + d^k - x^\star\Vert }{\Vert x^k - x^\star \Vert}\right)^2 + \left(\frac{o(\Vert x^k + d^k - x^\star\Vert)}{\Vert x^k - x^\star \Vert}\right)^2}{\lambda_{\min}(\nabla^2 f(x^\star)) + \left(\frac{o(\Vert x^k - x^\star\Vert)}{\Vert x^k - x^\star \Vert}\right)^2} \to 0 \qquad \text{as } k \to \infty.
        \end{aligned}
    \end{equation*}
    Since \(x^\star\) is a strong local minimum, and since \(x^k + p^k \to x^\star\), eventually we must have \(f(x^k + p^k) \geq f(x^\star)\). 
    From the anisotropic descent inequality we obtain that
    \begin{equation*}
        f(x^k) - f(x^\star) \geq f(x^k) - f(x^k + p^k) \geq \gamma \phi(\nabla \phi^*(\nabla f(x^k))).
    \end{equation*}
    Consequently,
    \begin{equation*}
        \begin{aligned}
            f(x^k + d^k) - f(x^k) &= - (1 - \varepsilon_k) (f(x^k) - f(x^\star))\\
            &\leq - (1 - \varepsilon_k) \gamma \phi(\nabla \phi^*(\nabla f(x^k)))\\
            &\leq - \gamma \sigma \phi(\nabla \phi^*(\nabla f(x^k)))
        \end{aligned}
    \end{equation*}
    where the last step follows from \(\varepsilon_k \to 0\) and \(\sigma < 1\) so that eventually \(1 - \varepsilon \geq \sigma\).
    We conclude that eventually (for large enough \(k\)) the linesearch condition holds and unit stepsize \(\tau_k = 1\) is always accepted. 
    \Cref{alg:globalized-pn} then reduces to \(x^{k+1} = x^k + d^k\), where \(d^k\) satisfies \eqref{eq:pn}. 
\end{proof}

%% file: appendix/cubic-subproblem-well-definedness.tex
\subsection{Proof of \cref{thm:cubic-subproblem-well-definedness}}

Denote by \(\xi_1 \leq \dots \leq \xi_n\) the generalized eigenvalues of the pencil \((A, M)\), 
and by \(v_i\) for \(1 \leq i \leq n\) the corresponding generalized eigenvectors. 
The diagonal matrix \(\Xi\) has the generalized eigenvalues as diagonal elements, \(
V\) is the matrix whose columns are the generalized eigenvectors, i.e.,
\begin{equation*}
    A V = M V \Xi, \qquad \text{where} \qquad V^\top M V = \id.
\end{equation*}
We first present two auxiliary lemmas.
\begin{lemma} \label{lem:V-invert}
    \(V\) is invertible with \(V^{-1} = V^\top M\).
\end{lemma}
\begin{proof}
    Starting from \(V^\top M V = \id\) and applying the determinant yields 
    \begin{equation*}
        \det(V^\top M V) = \det(\id) = 1.
    \end{equation*}
    From the multiplicative property of determinants it follows that
    \begin{equation*}
        \det(V^\top M V) = \det(V^\top) \det(M) \det(V) = \det(V)^2 \det(M).
    \end{equation*}
    Therefore,
    \begin{equation*}
        \det(V) = \frac{1}{\sqrt{\det(M)}} \neq 0,
    \end{equation*}
    where the inequality follows from \(\det M > 0\) (by positive definiteness).
    We conclude that \(V\) is invertible.
    The expression for \(V^{-1}\) follows by right-multiplying \(V^\top M V = \id\) by \(V^{-1}\) (or noting that the inverse is unique).
\end{proof}

\begin{lemma} \label{lem:cubic-psd}
    \(A+\lambda M\) is positive semidefinite if and only if \(\lambda \geq -\xi_1\).
\end{lemma}
\begin{proof}
    Positive semidefiniteness of \(A+ \lambda M\) means that
    \begin{equation*}
        \langle u, ( A + \lambda M ) u \rangle \geq 0
    \end{equation*}
    for all non-zero \(u \in \R^n\).
    Since \(V\) is invertible by \cref{lem:V-invert}, the mapping \(u = V y\) is a bijection. Substituting this into the quadratic form:
    \begin{align*}
        \langle u, ( A + \lambda M ) u \rangle &= \langle V y, ( A + \lambda M ) V y \rangle \\
        &= y^\top ( V^\top A V + \lambda V^\top M V ) y \\
        &= y^\top (\Xi + \lambda \id) y \\
        &= \sum_{i=1}^n (\xi_i + \lambda) y_i^2.
    \end{align*}
    This sum is non-negative for all \(y \neq 0\) if and only if \(\xi_i + \lambda \geq 0\) for all \(i\). Since the eigenvalues are ordered \(\xi_1 \leq \dots \leq \xi_n\), this is equivalent to \(\lambda \geq -\xi_1\).
\end{proof}

We are now ready to prove \cref{thm:cubic-subproblem-well-definedness}.
\begin{proof}
    By \cref{lem:cubic-psd}, we have that \(A + \lambda M \succeq 0\) if and only if \(\lambda \geq - \xi_1\).
    Since \(\lambda = \sigma \Vert s \Vert \geq 0\), the admissable domain is \(\lambda \in [\lambda_s, +\infty)\), where \(\lambda_s \eqdef \max \{-\xi_1, 0\}\).
    Thus any value of \(\lambda \in [\lambda_s, +\infty)\) satisfies \eqref{eq:cubic-subproblem-eq3}.

    For all \(\lambda > \lambda_s\), the matrix \(A+\lambda M\) is positive definite, so based on \eqref{eq:cubic-subproblem-eq1} we have that \(s(\lambda) = -(A + \lambda M)^{-1} M g\).
    Define
    \begin{equation*}
        \pi(\lambda) \eqdef \Vert s(\lambda) \Vert^2, \qquad \Psi(\lambda) \eqdef \pi(\lambda) - \frac{\lambda^2}{\sigma^2}.
    \end{equation*}
    Using the generalized eigenvalue decomposition of the pencil \((A, M)\), with generalized eigenvalues \(\xi_1, \leq \dots \leq \xi_n\) and corresponding eigenvectors \(v_1,\dots,v_n\),
    we obtain the expression
    \begin{equation} \label{eq:s-lambda-gev}
        s(\lambda) = - \sum_{i = 1}^n \frac{{g_V}_i}{\xi_i + \lambda} v_i, \qquad {g_V}_i \eqdef v_i^\top M g.
    \end{equation}
    Any pair \((s(\lambda), \lambda)\) satisfies \eqref{eq:cubic-subproblem-eq1}, and it remains to enforce \eqref{eq:cubic-subproblem-eq2} by finding a 
    \(\lambda\) for which \(\Psi(\lambda) = 0\).
    Define \(c(\lambda) \in \R^n\) as the vector with components \(c_i(\lambda) \eqdef \frac{{g_V}_i}{\xi_i + \lambda}\) for \(1 \leq i \leq n\), and \(G \in \R^{n \times n}\) as the Gram matrix with entries \(G_{ij} = v_i^\top v_j\) for \(1 \leq i,j \leq n\).
    Note that the Gram matrix \(G\) is positive definite by linear independence of the vectors \(v_i\), which follows by invertibility of \(V\) (cf.\,\cref{lem:V-invert}).
    Then we have that \(
        \pi(\lambda) = \Vert s(\lambda) \Vert^2 = c(\lambda)^\top G c(\lambda).
    \)
    It follows that \(\pi(\lambda)\) and therefore also \(\Psi(\lambda)\) are continuous on \((\lambda_s, +\infty)\).
    Moreover, by positive definiteness of \(G\) and since \(c(\lambda) \to 0\) as \(\lambda \to +\infty\) we have \[
        0 \leq \lim_{\lambda \to +\infty} \lambda_{\min}(G) \Vert c(\lambda) \Vert^2 \leq \lim_{\lambda \to +\infty} \pi(\lambda) \leq \lim_{\lambda \to +\infty} \lambda_{\max}(G) \Vert c(\lambda) \Vert^2 = 0.
    \]
    We conclude that \(\lim_{\lambda \to +\infty} \pi(\lambda) = 0\) and hence \(
        \lim_{\lambda \to +\infty} \Psi(\lambda) = \lim_{\lambda \to +\infty} - \frac{\lambda^2}{\sigma^2} = -\infty.
    \)
    We now distinguish three cases:
    
    \noindent
    \textbf{Case 1:} \(\xi_1 \geq 0\). 
    Clearly, \(\lambda_s = 0\).
    The limit of \(\Psi\) as \(\lambda \downto \lambda_s\)
    \begin{equation*}
        \lim_{\lambda \downto \lambda_s} \Psi(\lambda) = \lim_{\lambda \downto 0} \pi(\lambda) \geq \lim_{\lambda \downto 0} \lambda_{\min}(G) \Vert c(\lambda) \Vert^2 > 0
    \end{equation*}
    may be finite or infinite, but is never equal to zero. 
    Indeed, \(\lim_{\lambda \downto 0} \Vert c(\lambda)\Vert = 0\) if and only if \(g = 0\), which is excluded by assumption.
    Thus, by the intermediate value theorem there exists a \(\lambda_\star \in (\lambda_s, +\infty)\) for which \(\Psi(\lambda_\star) = 0\).
    This means that \(s^\star = s(\lambda_\star)\) as in \eqref{eq:s-lambda-gev} is well-defined, and that the pair \((s^\star, \lambda_\star)\) satisfies \eqref{eq:cubic-subproblem-eq1} - \eqref{eq:cubic-subproblem-eq3}.

    \noindent
    \textbf{Case 2:} \(\xi_1 < 0\) and there exists at least one \(i\) with \(\xi_i = \xi_1\) and \({g_V}_i \neq 0\). 
    Hence \(\lambda_s = -\xi_1 > 0\). 
    It follows that \[
        \lim_{\lambda \downto \lambda_s} \pi(\lambda) \geq \lim_{\lambda \downto \lambda_s} \lambda_{\min}(G) \Vert c(\lambda) \Vert^2 = + \infty.
    \]
    and hence that \(\lim_{\lambda \downto \lambda_s} \Psi(\lambda) = + \infty.\)
    The reasoning from Case 1 applies again.

    \noindent
    \textbf{Case 3:} \(\xi_1 < 0\) and \({g_V}_i = 0\) for all \(i \in \cI \eqdef \{1 \leq i \leq n \mid \xi_i = \xi_1\}\). 
    We have that \(\lambda_s = -\xi_1 > 0\), and the expression \eqref{eq:s-lambda-gev} for \(s(\lambda)\) can be written as
    \begin{equation*}
        s(\lambda) = - \sum_{\substack{i = 1\\i \notin \cI}}^n \frac{{g_V}_i}{\xi_i + \lambda} v_i
    \end{equation*}
    In this case the limit
    \begin{equation*}
        \lim_{\lambda \downto \lambda_s} \pi(\lambda) = \lim_{\lambda \downto \lambda_s} \Vert s(\lambda) \Vert^2 < + \infty.
    \end{equation*}
    This implies that \(\Psi(\lambda_s) \eqdef \lim_{\lambda \downto \lambda_s} \Psi(\lambda)\) is finite.
    If \(\Psi(\lambda_s) > 0\), then again the intermediate value theorem again guarantees the existence of a \(\lambda_\star \in (\lambda_s, +\infty)\) for which \(\Psi(\lambda_\star) = 0\) and for which \((s(\lambda_\star), \lambda_\star)\) satisfies \eqref{eq:cubic-subproblem-eq1} - \eqref{eq:cubic-subproblem-eq3}.
    
    Otherwise \(\Psi(\lambda_s) \leq 0\).
    This is the \emph{hard case}.
    The vector \(s_s \eqdef \lim_{\lambda \downto \lambda_s} s(\lambda)\) is well-defined and satisfies
    \(
        (A + \lambda_s M) s_s = - g
    \).
    However, \(\Psi(\lambda_s) \leq 0\) only ensures that \(\Vert s_s \Vert \leq \frac{\lambda_s}{\sigma}\), whereas equality is required.
    We remark that \(\xi_1 = -\lambda_s\) is a generalized eigenvalue of the pencil \((A, M)\) corresponding to the eigenvector \(v_1\), i.e., \((A + \lambda_s M) v_1 = 0\).
    Therefore, we have for any \(\alpha \in \R\) that
    \begin{equation*}
        (A + \lambda_s M) (s_s + \alpha v_1) = - Mg.
    \end{equation*}
    Note that \(\alpha \mapsto \Vert s_s + \alpha v_1 \Vert^2\) is continuous, quadratic and unbounded above.
    One can therefore always select an \(\alpha_s \in \R\) for which \(
        \Vert s_s + \alpha_s v_1 \Vert^2 = \frac{\lambda_s^2}{\sigma^2},
    \)
    and it then follows that the pair \((s_s + \alpha_s v_1, \lambda_s)\) satisfies \eqref{eq:cubic-subproblem-eq1} - \eqref{eq:cubic-subproblem-eq3}.
\end{proof}

%% file: appendix/cubic-lem-1.tex
\subsection{Proof of \cref{lem:cubic-lem-1}}

\begin{proof}
    We apply the fundamental theorem of calculus to the objective \(f\), i.e.,
    \begin{align*}
        f(x^k+s^k) - f(x^k) &= \int_{0}^{1} \langle \nabla f(x^k+ts^k), s^k \rangle \mathrm d t = \int_{0}^{1} \nu(x^k+ts^k) \langle g(x^k+ts^k), s^k \rangle \mathrm d t.
    \end{align*}
    We upper bound \(\langle g(x^k+ts^k), s^k \rangle\). The fundamental theorem of calculus applied to \(g\) yields
    \begin{equation*}
        g(x^k+ t s^k) = g(x^k) + \int_{0}^t H(x^k+\tau s^k) s^k \mathrm d \tau = g(x^k) + t H(x^k) s^k + r(t s^k),
    \end{equation*}
    where the residual term equals
    \begin{equation} \label{prf-cubic-lem-1-residual}
        r(t s^k) := \int_{0}^{t} \left(H(x^k+\tau s^k) - H(x^k)\right) s^k \mathrm d \tau.
    \end{equation}
    Since \(H(x^k) s^k = - g(x^k) - \lambda_k s^k\) by \eqref{eq:cubic-key-relation} it follows that
    \begin{equation*}
        g(x^k+ t s^k) = (1-t) g(x^k) - t \lambda_k s^k + r(t s^k).
    \end{equation*}
    Taking the inner product with \(s^k\) yields
    \begin{align*}
        \langle g(x^k+ t s^k), s^k \rangle &= (1-t) \langle g(x^k), s^k \rangle - t \lambda_k \Vert s^k \Vert^2 + \langle r(t s^k), s^k \rangle\\
        &\leq - t \lambda_k \Vert s^k \Vert^2 + \langle r(t s^k), s^k \rangle,
    \end{align*}
    where the inequality follows from the fact that \(1 - t \geq 0\) and \(\langle g(x^k), s^k \rangle = \frac{1}{\nu(x^k)} \langle \nabla f(x^k), s^k \rangle \leq 0\) (cf.\,\eqref{eq:s-descent}).
    Under \cref{as:prec-hessian-lipschitz} we can upper bound
    \begin{equation} \label{prf-cubic-lem-1-residual-bound}
        \Vert r(t s^k) \Vert \leq \int_0^t \left \Vert \left( H(x^k+\tau s^k) - H(x^k) \right) s^k \right \Vert \mathrm d \tau \leq \int_{0}^{t} L_H \tau \Vert s^k \Vert^2 \mathrm d \tau = \frac{L_H t^2}{2} \Vert s^k \Vert^2.
    \end{equation}
    From Cauchy-Schwarz and because \(\sigma \geq L_H\) we have that
    \begin{align*}
        \langle g(x^k + t s^k), s^k \rangle &\leq - t \lambda_k \Vert s^k \Vert^2 + \langle r(ts^k), s^k \rangle\\
        &\leq - t \lambda_k \Vert s^k \Vert^2 + \Vert r(t s^k) \Vert \Vert s^k \Vert\\
        &\leq - t \sigma \Vert s^k \Vert^3 + \frac{L_H t^2}{2} \Vert s^k \Vert^3\\
        &\leq - \left( t L_H - \frac{L_H t^2}{2} \right) \Vert s^k \Vert^3.
    \end{align*}
    Thus, using \(\nu(z) \geq \widetilde L\) for all \(z \in \R^n\) we obtain
    \begin{align*}
        f(x^k+s^k) - f(x^k) &= \int_{0}^{1} \nu(x^k + t s^k) \langle g(x^k + t s^k), s^k \rangle \mathrm d t\\
        &\leq - \widetilde L \int_{0}^{1} \left( t L_H - \frac{L_H t^2}{2} \right) \Vert s^k \Vert^3 \mathrm d t\\
        &\leq -\widetilde L L_H \Vert s^k \Vert^3 \int_{0}^{1} \left(t - \frac{t^2}{2} \right) \mathrm d t\\
        &= -\widetilde L L_H \Vert s^k \Vert^3 \left(\frac{1}{2} - \frac{1}{6} \right) = - \frac{\widetilde L L_H}{3} \Vert s^k \Vert^3.
    \end{align*}
\end{proof}

%% file: appendix/cubic-lem-2.tex
\subsection{Proof of \cref{lem:cubic-lem-2}}

\begin{proof}
    By the fundamental theorem of calculus we have \(g(x^k + s^k) = g(x^k) + H(x^k) s^k + r(s^k)\), where \(r\) is given by \eqref{prf-cubic-lem-1-residual}.
    Combined with \eqref{eq:cubic-key-relation} and \(\lambda_k = \sigma \Vert s^k \Vert\) this yields
    \begin{align*}
        g(x^{k+1}) &= g(x^k + s^k) = g(x^k) + H(x^k) s^k + r(s^k)\\
        &= -\lambda_k s^k + r(s^k) = -\sigma \Vert s^k \Vert s^k + r(s^k).
    \end{align*}
    The claim then follows from the triangle inequality and \eqref{prf-cubic-lem-1-residual-bound}:
    \begin{align*}
        \Vert g(x^{k+1}) \Vert &\leq \sigma \Vert s^k \Vert^2 + \frac{L_H}{2} \Vert s^k \Vert^2 = \frac{2\sigma + L_H}{2} \Vert s^k \Vert^2.
    \end{align*}
\end{proof}

%% file: appendix/cubic-pn-convergence-rate.tex
\subsection{Proof of \cref{thm:cubic-pn-convergence-rate}}

\begin{proof}
    \Cref{lem:cubic-lem-2} and \(\sigma \geq L_H\) ensure that
    \begin{equation*}
        \Vert s^k \Vert \geq \sqrt{\frac{2}{3 \sigma} \Vert g(x^{k+1}) \Vert}.
    \end{equation*}
    Combined with \cref{lem:cubic-lem-1} and because \(\left(\frac{2}{3}\right)^{\nicefrac{3}{2}} \geq \frac{1}{2}\), this yields 
    \begin{align*}
        f(x^{k+1}) - f(x^k) &\leq - \frac{\widetilde L L_H}{3} \left( \frac{2}{3 \sigma} \Vert g(x^{k+1}) \Vert \right)^{\nicefrac{3}{2}} = - \frac{\widetilde L L_H}{3 \sigma^{\nicefrac{3}{2}}} \left(\frac{2}{3}\right)^{\nicefrac{3}{2}} \Vert g(x^{k+1}) \Vert^{\nicefrac{3}{2}}\\
        &\leq - \frac{\widetilde L L_H}{6 \sigma^{\nicefrac{3}{2}}} \Vert g(x^{k+1}) \Vert^{\nicefrac{3}{2}}.
    \end{align*}
    Suppose that the first \(K-1\) iterations do not yield an \(\varepsilon\)-stationary point, i.e., \(\Vert g(x^{k+1}) \Vert \geq \varepsilon\) for all \(k \leq K-1\).
    Then a standard telescoping argument yields
    \begin{align*}
        f(x^0) - \inf f &\geq f(x^0) - f(x^{K}) = \sum_{k = 0}^{K-1} \left(f(x^k) - f(x^{k+1})\right)\\
        &\geq \sum_{k = 0}^{K-1} \left( \frac{\widetilde L L_H}{6 \sigma^{\nicefrac{3}{2}}} \Vert g(x^{k+1}) \Vert^{\nicefrac{3}{2}}\right) \geq \sum_{k = 0}^{K-1} \left( \frac{\widetilde L L_H}{6 \sigma^{\nicefrac{3}{2}}}\varepsilon^{\nicefrac{3}{2}}\right) = K \left( \frac{\widetilde L L_H}{6 \sigma^{\nicefrac{3}{2}}} \varepsilon^{\nicefrac{3}{2}}\right).
    \end{align*}
    It immediately follows that 
    \begin{equation*}
        K \leq \frac{6 \sigma^{\nicefrac{3}{2}} (f(x^0) - \inf f)}{\widetilde L L_H \varepsilon^{\nicefrac{3}{2}}}.
    \end{equation*}
    This proves the claim.
\end{proof}

%% file: appendix/adaptive-cubic-pn-convergence-rate.tex
\subsection{Proof of \cref{thm:adaptive-cubic-pn-convergence-rate}}

The following lemma serves as a counterpart to \cref{lem:cubic-lem-1} and ensures sufficient decrease.
\begin{lemma} \label{lem:adaptive-cubic-pn-lem-1}
    Suppose that \cref{as:prec-hessian-lipschitz,ass:isotropic-reference} hold. If \(\sigma_k \geq L_H + \theta\), then
    \begin{equation*}
         f(x^k+s^k) - f(x^k) \leq - \frac{\widetilde L \sigma_k}{4} \Vert s^k \Vert^3.
    \end{equation*}
\end{lemma}

\input{appendix/adaptive-cubic-pn-lem-1.tex}

Under the conditions of \cref{lem:adaptive-cubic-pn-lem-1} we have \(\rho_k \geq 1\).
From the update rule \eqref{eq:sigma-update} we can deduce that \(\sigma_k\) remains bounded.
\begin{lemma} \label{lem:sigma-upper-bound}
    Suppose that \cref{as:prec-hessian-lipschitz,ass:isotropic-reference} hold.
    Then for all iterations \(k\) of \cref{alg:pn-adaptive-cubic-regularization},
    \begin{equation}
        \sigma_k \leq \sigma_{\max} \eqdef \max \left( \sigma_0, \gamma_3 (L_H + \theta) \right)
    \end{equation}
\end{lemma}

\input{appendix/sigma-upper-bound.tex}

As before, we lower bound \(\Vert s^k \Vert\) in terms of \(\Vert g(x^{k+1}) \Vert\). 
\begin{lemma} \label{lem:adaptive-cubic-pn-lem-2}
    Suppose that \cref{as:prec-hessian-lipschitz,ass:isotropic-reference} holds.
    If iteration \(k\) of \cref{alg:pn-adaptive-cubic-regularization} is successful, then
    \begin{equation*}
        \Vert s^k \Vert^2 \geq \frac{2}{2\sigma_{\max} + L_H + \theta} \Vert g(x^{k+1}) \Vert.
    \end{equation*}
\end{lemma}

\input{appendix/adaptive-cubic-pn-lem-2.tex}

Denote by \(\cS_k \subseteq \{i \in \N \mid 0 \leq i \leq k\}\) the set of indices corresponding to successful iterations, i.e., for which $\rho_k \geq \eta_1$.
We first present some auxiliary results.
The proof of the following lemma follows a strategy similar to that of \cref{thm:cubic-pn-convergence-rate}. 
\begin{lemma} \label{lem:adaptive-cubic-successful-iterations}
    If \cref{as:prec-hessian-lipschitz,ass:isotropic-reference} hold, then \cref{alg:pn-adaptive-cubic-regularization} 
    requires at most
    \begin{equation*}
        \vert \cS_K \vert \leq \frac{4}{\widetilde L \sigma_{\min}} \left(\frac{2 \sigma_{\max} + L_H + \theta}{2}\right)^{\nicefrac{3}{2}} \frac{f(x^0) - \inf f}{\varepsilon^{\nicefrac{3}{2}}} + 1
    \end{equation*}
    successful iterations to find a point such that \(\Vert \nabla \phi^*(\nabla f(x^k)) \Vert \leq \varepsilon\).
\end{lemma}
\begin{proof}
    \Cref{lem:adaptive-cubic-pn-lem-2} ensures that
    \begin{equation*}
        \Vert s^k \Vert \geq \sqrt{\frac{2}{2\sigma_{\max} + L_H + \theta} \Vert g(x^{k+1}) \Vert}.
    \end{equation*}
    By definition of $\rho_k$, this yields for any successful iteration \(k \in \cS_{K-1}\)
    \begin{align*}
        f(x^{k+1}) - f(x^k) &\leq - \eta_1 \frac{\widetilde L \sigma_k}{4} \left( \frac{2}{2\sigma_{\max} + L_H + \theta} \Vert g(x^{k+1}) \Vert \right)^{\nicefrac{3}{2}}\\
        &\leq - \eta_1 \frac{\widetilde L \sigma_{\min}}{4} \left( \frac{2}{2\sigma_{\max} + L_H + \theta} \right)^{\nicefrac{3}{2}} \Vert g(x^{k+1}) \Vert^{\nicefrac{3}{2}}.
    \end{align*}
    Assume, without loss of generality that the first iteration is successful, and suppose that the first \(K-1\) iterations do not yield an \(\varepsilon\)-stationary point, i.e., \(\Vert g(x^{k+1}) \Vert \geq \varepsilon\) for all \(k \leq K-1\).
    Then a standard telescoping argument yields
    \begin{align*}
        f(x^0) - \inf f &\geq f(x^0) - f(x^{K}) = \sum_{k \in \cS_{K-1}} \left(f(x^k) - f(x^{k+1})\right)\\
        &\geq \sum_{k \in \cS_{K-1}} \left(  \eta_1 \tfrac{\widetilde L \sigma_{\min}}{4} \left( \tfrac{2}{2\sigma_{\max} + L_H + \theta} \right)^{\nicefrac{3}{2}} \Vert g(x^{k+1}) \Vert^{\nicefrac{3}{2}}\right)\\
        &\geq \sum_{k \in \cS_{K-1}} \left( \eta_1 \tfrac{\widetilde L \sigma_{\min}}{4} \left( \tfrac{2}{2\sigma_{\max} + L_H + \theta} \right)^{\nicefrac{3}{2}} \varepsilon^{\nicefrac{3}{2}}\right)\\
        &= \lvert \cS_{K-1} \rvert \left( \eta_1 \tfrac{\widetilde L \sigma_{\min}}{4} \left( \tfrac{2}{2\sigma_{\max} + L_H + \theta} \right)^{\nicefrac{3}{2}} \varepsilon^{\nicefrac{3}{2}}\right).
    \end{align*}
    It immediately follows from the observation that \(\vert \cS_K \vert \leq \vert \cS_{K-1} \vert + 1\) that 
    \begin{equation*}
        \lvert \cS_{K} \rvert \leq \frac{4}{\eta_1 \widetilde L \sigma_{\min}} \left(\frac{2 \sigma_{\max} + L_H + \theta}{2}\right)^{\nicefrac{3}{2}} \frac{f(x^0) - \inf f}{\varepsilon^{\nicefrac{3}{2}}} + 1.
    \end{equation*}
    This proves the claim.
\end{proof}
The update rule \eqref{eq:sigma-update} and the upper bound \(\sigma_k \leq \sigma_{\max}\) from \cref{lem:sigma-upper-bound} imply an upper bound of the total number of iterations in terms of the number of successful number of iterations.
\begin{lemma}[{\cite[Lemma 2.4.1]{cartis2022evaluation}}] \label{lem:iterations-successful-bound}
    Suppose that \(\sigma_k\) is updated according to \eqref{eq:sigma-update}, and that \(\sigma_k \leq \sigma_{\max}\) for some \(\sigma_{\max} > 0\).
    Then,
    \begin{equation}
        k \leq \vert \cS_k \vert \left( 1 + \frac{\vert \log \gamma_1 \vert}{\log \gamma_2} \right) + \frac{1}{\log \gamma_2} \log \left( \frac{\sigma_{\max}}{\sigma_0} \right).
    \end{equation}
\end{lemma}
We can now prove \cref{thm:adaptive-cubic-pn-convergence-rate}.
\begin{proof}
    This is a direct combination of \cref{lem:adaptive-cubic-successful-iterations,lem:iterations-successful-bound}.
\end{proof}

%% file: appendix/adaptive-cubic-pn-lem-1.tex
\begin{proof}
    The proof is similar to that of \cref{lem:cubic-lem-1}.

    We apply the fundamental theorem of calculus to the objective \(f\), i.e.,
    \begin{align*}
        f(x^k+s^k) - f(x^k) &= \int_{0}^{1} \langle \nabla f(x^k+ts^k), s^k \rangle \mathrm d t = \int_{0}^{1} \nu(x^k+ts^k) \langle g(x^k+ts^k), s^k \rangle \mathrm d t.
    \end{align*}
    We upper bound \(\langle g(x^k+ts^k), s^k \rangle\). The fundamental theorem of calculus applied to \(g\) yields
    \begin{equation*}
        g(x^k+ t s^k) = g(x^k) + \int_{0}^t H(x^k+\tau s^k) s^k \mathrm d \tau = g(x^k) + t H(x^k) s^k + r(t s^k),
    \end{equation*}
    where the residual term equals
    \begin{equation*}
        r(t s^k) := \int_{0}^{t} \left(H(x^k+\tau s^k) - H(x^k)\right) s^k \mathrm d \tau.
    \end{equation*}
    Since \(H(x^k) s^k = - g(x^k) - \lambda_k s^k + z^k\) it follows that
    \begin{equation*}
        g(x^k+ t s^k) = (1-t) g(x^k) - t \lambda_k s^k + r(t s^k) + t z^k.
    \end{equation*}
    Taking the inner product with \(s^k\) yields
    \begin{align*}
        \langle g(x^k+ t s^k), s^k \rangle &= (1-t) \langle g(x^k), s^k \rangle - t \lambda_k \Vert s^k \Vert^2 + \langle r(t s^k), s^k \rangle + t \langle z^k, s^k \rangle\\
        &\leq - t \lambda_k \Vert s^k \Vert^2 + \langle r(t s^k), s^k \rangle + t \langle z^k, s^k \rangle,
    \end{align*}
    where the inequality follows from the fact that \(1 - t \geq 0\) and \(\langle g(x^k), s^k \rangle \leq 0\).
    Under \cref{as:prec-hessian-lipschitz} we can upper bound
    \begin{equation*}
        \Vert r(t s^k) \Vert \leq \int_0^t \left \Vert \left( H(x^k+\tau s^k) - H(x^k) \right) s^k \mathrm d \tau \right \Vert \leq \int_{0}^{t} L_H \tau \Vert s^k \Vert^2 \mathrm d \tau = \frac{L_H t^2}{2} \Vert s^k \Vert^2.
    \end{equation*}
    In combination with \(\Vert z^k \Vert \leq \frac{1}{2} \theta \Vert s^k \Vert^2\) and \(\sigma_k \geq \theta + L_H \geq \theta + t L_H\) for \(t \in [0, 1]\) this yields
    \begin{align*}
        \langle g(x^k + t s^k), s^k \rangle &\leq - t \lambda_k \Vert s^k \Vert^2 + \langle r(ts^k), s^k \rangle + t \langle z^k, s^k \rangle\\
        &\leq - t \lambda_k \Vert s^k \Vert^2 + \Vert r(t s^k) \Vert \Vert s^k \Vert + t \Vert z^k \Vert \Vert s^k \Vert\\
        &\leq - t \sigma_k \Vert s^k \Vert^3 + \frac{L_H t^2}{2} \Vert s^k \Vert^3 + \frac{t \theta}{2} \Vert s^k \Vert^3\\
        &\leq - t\left( \sigma_k - \frac{1}{2} \left( t L_H + \theta \right) \right) \Vert s^k \Vert^3\\
        &\leq - t \frac{\sigma_k}{2} \Vert s^k \Vert^3.
    \end{align*}
    Thus, we obtain
    \begin{align*}
        f(x^k+s^k) - f(x^k) &= \int_{0}^{1} \nu(x^k + t s^k) \langle g(x^k + t s^k), s^k \rangle \mathrm d t \leq - \widetilde L \int_{0}^{1} t \frac{\sigma_k}{2} \Vert s^k \Vert^3 \mathrm d t\\
        &\leq -\widetilde L \frac{\sigma_k}{2} \Vert s^k \Vert^3 \int_{0}^{1} t \mathrm d t = -\widetilde L \frac{\sigma_k}{4} \Vert s^k \Vert^3.
    \end{align*}
\end{proof}

%% file: appendix/sigma-upper-bound.tex
\begin{proof}
    If \(\sigma_k \geq L_H + \theta\), then \(\rho_k \geq 1\) by \cref{lem:adaptive-cubic-pn-lem-1}. Thus the mechanism of \cref{alg:pn-adaptive-cubic-regularization} implies the claimed bound (\(\sigma_{k-1}\) is increased by at most a factor \(\gamma_3\) and only if \(\sigma_{k-1} \le \eta_1 < 1\)). 
\end{proof}

%% file: appendix/adaptive-cubic-pn-lem-2.tex
\begin{proof}
    The proof is similar to that of \cref{lem:cubic-lem-2}.
    By the fundamental theorem of calculus we have \(g(x^k + s^k) = g(x^k) + H(x^k) s^k + r(s^k)\), where \(r\) is given by \eqref{prf-cubic-lem-1-residual}.
    Combined with \(H(x^k) s^k + \lambda_k s^k = - g(x^k) + z^k\) and \(\lambda_k = \sigma \Vert s^k \Vert\) this yields
    \begin{align*}
        g(x^{k+1}) &= g(x^k + s^k) = g(x^k) + H(x^k) s^k + r(s^k)\\
        &= -\lambda_k s^k + r(s^k) + z^k = -\sigma_k \Vert s^k \Vert s^k + r(s^k) + z^k.
    \end{align*}
    By the triangle inequality and \(\Vert z^k \Vert \leq \frac{\theta}{2} \Vert s^k \Vert\) we obtain
    \begin{align*}
        \Vert g(x^{k+1}) \Vert &\leq \sigma_k \Vert s^k \Vert^2 + \frac{L_H}{2} \Vert s^k \Vert^2 + \frac{\theta}{2} \Vert s^k \Vert = \frac{2\sigma_k + L_H + \theta}{2} \Vert s^k \Vert^2.
    \end{align*}
    The claim follows by \cref{lem:sigma-upper-bound}.
\end{proof}